\documentclass[a4paper,11pt]{amsart}

\usepackage[hang,flushmargin]{footmisc}

\usepackage[foot]{amsaddr}
\usepackage{hyperref}

\usepackage{graphicx,rotating}
\usepackage{subfigure}
\usepackage[font=footnotesize,labelfont=footnotesize]{caption}
\usepackage[utf8]{inputenc}
\usepackage[toc]{appendix}
\usepackage{twoopt}
\usepackage{float}
\usepackage{orcidlink}
\usepackage{amssymb}
\usepackage{amsthm}
\usepackage{amsmath,a4wide}
\numberwithin{equation}{section}
\usepackage{mathtools}
\usepackage{amsfonts}
\usepackage{marvosym}
\usepackage{verbatim}
\usepackage{marginnote}
\usepackage{enumerate}
\usepackage{wrapfig}
\usepackage{enumerate}
\mathtoolsset{showonlyrefs}

\theoremstyle{plain}
\newtheorem{theorem}{Theorem}[section]
\newtheorem*{theorem*}{Theorem}
\theoremstyle{plain}
\newtheorem{corollary}[theorem]{Corollary}
\theoremstyle{plain}
\newtheorem{lemma}[theorem]{Lemma}

\theoremstyle{plain}
\newtheorem{proposition}[theorem]{Proposition}
\newtheorem*{proposition*}{Proposition}
\theoremstyle{definition}

\theoremstyle{remark}

\theoremstyle{remark}

\theoremstyle{definition}

\theoremstyle{plain}

\theoremstyle{definition}

\providecommand{\norm}[1]{\left\lVert #1 \right\rVert}
\providecommand{\abs}[1]{\left\lvert #1 \right\rvert}

\newcommand{\R}{\mathbb{R}}

\newcommand{\Z}{\mathbb{Z}}
\newcommand{\N}{\mathbb{N}}
\newcommand{\Lt}[1][d]{L^2(\R^{#1})}
\newcommand{\G}{\mathcal{G}}
\newcommand{\F}{\mathcal{F}}

\newcommand{\indicator}{\raisebox{2pt}{$\chi$}}
\renewcommand{\l}{\lambda}
\renewcommand{\L}{\Lambda}

\newcommandtwoopt{\xarrow}[2][0.5cm][0]{\mathrel{\rotatebox[origin=c]{#2}{$\xrightarrow{\rule{#1}{0pt}}$}}}

\DeclareMathOperator{\sech}{sech}
\DeclareMathOperator{\csch}{csch}
\DeclareMathOperator{\arccosh}{arccosh}
\DeclareMathOperator{\arcsinh}{arcsinh}

\usepackage{xcolor}

\allowdisplaybreaks

\newcommand{\suppfoot}{The full computation is provided in the supplement.}
\begin{document}

\title[Gabor frame bound optimizations]{Gabor frame bound optimizations}

\author[M.\ Faulhuber]{Markus Faulhuber \orcidlink{0000-0002-7576-5724}
}
\email{markus.faulhuber@univie.ac.at}
\author[I.\ Shafkulovska]{Irina Shafkulovska \orcidlink{0000-0003-1675-3122}
}
\email{irina.shafkulovska@univie.ac.at}
\address{NuHAG, Faculty of Mathematics, University of Vienna \newline Oskar-Morgenstern-Platz 1, 1090 Vienna, Austria}

\thanks{Markus Faulhuber was supported by the Austrian Science Fund (FWF) projects TAI6 and P33217. Irina Shafkulovska was supported by the Austrian Science Fund (FWF) project P33217. Parts of this work were established while Irina Shafkulovska was associated with the Acoustics Research Institute (ARI) at the Austrian Academy of Sciences (\"OAW)}

\begin{abstract}
	We study sharp frame bounds of Gabor systems over rectangular lattices for different windows and integer oversampling rate. In some cases we obtain optimality results for the square lattice, while in other cases the lattices optimizing the frame bounds and the condition number are rectangular lattices which are different for the respective quantities. Also, in some cases optimal lattices do not exist at all and a degenerated system is optimal.
\end{abstract}



\vspace*{7pt}
\maketitle

\section{Introduction}\label{sec_Intro}

We aim to find extremal lattices for the spectral bounds of Gabor systems with specific windows. The quantities which we aim to optimize are the lower and upper frame bound as well as their ratio, which is the condition number of the associated frame operator. We study the cases provided in \cite{Jan96}, where Janssen computed sharp spectral bounds for Gabor frames over rectangular lattices of the form $a \Z \times b \Z$ for several different window functions and $(a b)^{-1} \in \N$, and we will build upon this work. Our findings are for rectangular lattices of integer density and for different windows. The main results can be summed up as follows:

\medskip

\begin{center}
\fbox{
	\parbox{.9\textwidth}{
		$\bullet$ \textit{Hyperbolic secant}: the square lattice optimizes the lower and upper frame bound simultaneously and, hence, also the condition number.
		
		\medskip
		
		$\bullet$ \textit{Cut-off exponentials}: we find that a lattice optimizing the frame bounds does not exist for cut-off exponentials supported on $[0,1/b]$. For support $[0,2/b]$, optimal lattices may or may not exist, depending on the lattice density and which quantity we seek to optimize. If they exist, then they depend on the density and on the decay parameter of the exponential.
		
		\medskip
		
		$\bullet$ \textit{One-sided exponential}: the optimizing lattices for the condition number and the frame bounds are different from each other and depend on the (over)sampling rate.
		
		\medskip
		
		$\bullet$ \textit{Two-sided exponentials}: each quantity has a unique optimizer. They depend on the (over)sampling rate and they all differ from each other.
	}
}
\end{center}

Until now, Gaussians were the only family of window functions for which optimality results have been known \cite{BetFauSte21}, \cite{Faulhuber_Hexagonal_2018}, \cite{FauSte17}. We see that many of the situations are quite different from what we know about the Gaussian case. Therefore, our results emphasize that in general we may not expect that an optimal lattice exists at all (for the frame bounds or the condition number) and that a lattice optimizing one quantity does not necessarily optimize the others. Indeed, assume that the lower frame bound $A$ and the upper frame bound $B$ each have a unique critical point with respect to one free lattice parameter (others are fixed), so there exists a unique optimizing lattice. Critical points of the condition number $B/A$ need to fulfill
\begin{equation}\label{eq:A'=B'}
	\left( \frac{B}{A} \right)' = 0
	\quad \Longleftrightarrow \quad
	\frac{A'}{A} = \frac{B'}{B},
\end{equation}
where the differentiation is with respect to the free parameter.
If the critical points of $A$ and $B$ do not coincide, then no critical point of $B/A$ can coincide with either of the other two. As can be seen from the example of the two-sided exponential, this is not necessarily a pathology of the window not being in the modulation space $M^1(\R)$, also known as Feichtinger's algebra $S_0(\R)$. However, the situation for the hyperbolic secant in the rectangular case is comparable to the Gaussian case (see also \cite{JanStr02} for a connection).

The rigorous analytic study of optimal lattices for Gabor systems is relatively new and, to the best of our knowledge, has only been carried out numerically prior to \cite{FauSte17}. We are choosing a pedestrian approach for the analysis of the frame bounds' behavior, mainly using a critical point approach combined with careful interval estimates. The computations turn out to be rather technical and to avoid an overload of elementary (though delicate) calculations and estimates, we do not always give all details. However, the manuscript is self-contained and we also provide a supplement to this manuscript. We put numbered markers which tell where to find the respective calculation in the supplement.

This work is structured as follows. We present our results in Section \ref{sec:results}. In Section \ref{sec:notation} we settle the notation and provide some background information and motivation. Section \ref{sec:prep} contains explanation on auxiliary techniques used in the manuscript. The proofs of our results follow in Section \ref{sec:sech} for the hyperbolic secant, Section \ref{sec:cut} for cut-off exponentials, Section \ref{sec:1sided} for one-sided exponentials and, finally, in Section \ref{sec:2sided} we present the proofs for two-sided exponentials.

\section{The results}\label{sec:results}
We will now present an overview of the results and also provide some background information. The complete proofs are given in the later sections of this work. All parameters are always assumed to be positive, in particular $(a,b) \in \R_+ \times \R_+$.
\subsection{Hyperbolic secant}
The frame property of the hyperbolic secant and lattices of the form $a \Z \times b \Z$ has been studied by Janssen and Strohmer \cite{JanStr02}. By relating its Zak transform to the one of the Gaussian
they were able to show that
\begin{equation}
	\G(\sech, a \Z \times b \Z) \text{ is a frame}
	\quad \Longleftrightarrow \quad
	(ab)^{-1} > 1.
\end{equation}
As the (properly scaled) hyperbolic secant is invariant under the Fourier transform, it follows by the general theory of symplectic matrices and metaplectic operators \cite{Fol89, Gos11, Gro01} that the two Gabor systems
\begin{equation}\label{eq:G_sech}
	\G(\sech(\pi t), a \Z \times b \Z)
	\quad \text{ and } \quad
	\G(\sech(\pi t), b \Z \times a \Z)
	\quad \text{ are unitarily equivalent.}
\end{equation}
In particular, \eqref{eq:G_sech} says that the lattice parameters $(a,b)$ and $(b,a)$ yield the same frame bounds for the hyperbolic secant window $\sech(\pi t)$ and that the associated frame operator is invariant under rotations of the lattice by integer multiples of 90 degrees. Therefore, \emph{for any fixed density} $n$, the frame bounds as functions on $\{\L_{a,b}\mid (ab)^{-1} = n\}$ are symmetric about the pair $(1/\sqrt{n}, 1/\sqrt{n})$, so the square lattice has to be a \textit{local} extremum for both frame bounds. 
Theorem \ref{thm:sech} states that it is actually the unique \textit{global} optimum for integer oversampling.
\begin{theorem}\label{thm:sech}
	Let $g(t) = \left( \frac{\pi}{2} \right)^{1/2} \sech(\pi t)$ be the normalized and Fourier-invariant hyperbolic secant and consider the rectangular lattice $\L_{a, b} = a \Z \times b \Z$ of density $n$, i.e.~$(a b)^{-1} =  n$, with $2 \leq n \in \N$. By $A(a, b)$ and $B(a, b)$ we denote the sharp lower and upper frame bound of the Gabor system $\G (g, \L_{a, b})$, respectively. Then, we have that
	\begin{equation}
		A \left( \tfrac{1}{\sqrt{n}}, \tfrac{1}{\sqrt{n}} \right) \geq A (a,b)
		\quad \text{ and } \quad
		B \left( \tfrac{1}{\sqrt{n}}, \tfrac{1}{\sqrt{ n}} \right) \leq B (a,b),
	\end{equation}
	with equality if and only if $a = b = \tfrac{1}{\sqrt{n}}$, i.e., if the lattice is the square lattice of density $n$.
\end{theorem}
A simple corollary of Theorem \ref{thm:sech} can be obtained by a symplectic deformation result, which is in particular a dilation result (see Proposition \ref{prop:symplectic_invariance})
\begin{corollary}\label{cor:sech}
	Let $g_\gamma(t) = \left( \frac{\pi}{2 \gamma} \right)^{1/2} \sech\left(\frac{\pi t}{\gamma}\right)$ be the hyperbolic secant dilated by $\gamma$. Then, the optimizing lattice is $\frac{\gamma}{\sqrt{n}} \Z \times \frac{1}{\gamma \sqrt{n}} \Z$, i.e., the by $\gamma$ dilated version of the square lattice.
\end{corollary}

We remark that in the case of critical sampling, i.e., $(a b)^{-1} = 1$, we obtain
\begin{equation}
	B(\gamma,\tfrac{1}{\gamma}) \leq B(a\gamma,\tfrac{b}{\gamma}) < \infty.
\end{equation}
with equality if and only if $a=b=1$. This can be deduced from the explicit expressions for $B$ and combining results on Laplace-Beltrami operators on a flat torus \cite{Faulhuber_Determinants_2018} and results on Gaussian Gabor systems \cite{FauSte17}. Since we have a finite upper frame bound, the lower frame bound must vanish identically in this case, as imposed by the Balian-Low theorem.

\subsection{Cut-off exponentials}
The frame set of this class of functions is not known and so we focus only on the cases from \cite{Jan96}. In these situations we also know that we have a frame. We consider the following families of functions
\begin{equation}
	g_m(t) = C_{b,\gamma} e^{-\gamma t} \indicator_{[0,m/b]}(t), \quad m \in \{1,2\}, \, \gamma \geq 0,
	\quad C_{b,\gamma} \text{ chosen such that } \norm{g_m}_2=1.
\end{equation}

\subsubsection{The case \texorpdfstring{$m=1$}{m=1}} 

This particular family of windows is actually adapted to the lattice and we may perform an analysis depending on the decay parameter $\gamma>0$.
\begin{proposition}\label{pro:m1gamma}
	Let $A(\gamma)$ and $B(\gamma)$ be the frame bounds of the Gabor system $\G(g_1,a \Z \times b \Z)$ with $(a b)^{-1} = n \in \N$. Then, for any fixed lattice (hence also fixed $n$), the behavior of the frame bounds with respect to the decay parameter $\gamma > 0$ is
	\begin{equation}
		A(\gamma) < n, \; A'(\gamma) < 0
		\quad \text{ and } \quad
		B(\gamma) > n, \; B'(\gamma) > 0.
	\end{equation}
	The frame condition number is independent of $n$ and always given by
	\begin{equation}
		\frac{B(\gamma)}{A(\gamma)}= e^{2 a \gamma}.
	\end{equation}
\end{proposition}

From the calculations in the proof it will be obvious that the optimal window (for fixed $a$) is actually obtained for $\gamma \to 0$, which yields a tight frame with the box function. Moreover, the roles of $\gamma$ and $a$ are exchangeable. We may fix $\gamma > 0$ and then optimize with respect to $a$ and the above results hold verbatim with $\gamma$ replaced by $a$. Note that we then change the window and the lattice at the same time. The optimal system is degenerated $(a\to0, b\to\infty)$ and the window tends to a Dirac delta. 

\subsubsection{The case \texorpdfstring{$m=2$}{m=2}} 
We start with results for fixed lattices and optimization with respect to the decay parameter.
\begin{proposition}\label{pro:m2_gamma}
	Let $A(\gamma)$ and $B(\gamma)$ be the frame bounds of the Gabor system $\G(g_2,a \Z \times b \Z)$ with $(a b)^{-1} = n \in \N$. Then, for any fixed lattice (hence also fixed $n$), the behavior of the frame bounds with respect to the decay parameter $\gamma > 0$ is as follows:
	\begin{enumerate}[(i)]
		\item The lower frame bound $A(\gamma)$ has a unique maximum which depends on $n$ and $a = 1/(b n)$. The unique maximum is attained for the unique positive solution of the equation
		\begin{equation}
			-1 + \frac{1}{\gamma a}- \coth(\gamma a) = n \, \coth\left( \frac{\gamma n a}{2}\right) \sech(\gamma n a).
		\end{equation}
		\item For $n\in\{1,2\}$, we have $B(\gamma) > 2$ and $B'(\gamma) > 0$.
		\item For $3 \leq n \in \N$, we have that $B$ has a local maximum and a local minimum, which may be global (depending on $n$ and $a$). Denoting the point of the local maximum and minimum by $\gamma^*$ and $\gamma_*$, respectively, we have
		\begin{equation}
			\gamma^* < \frac{\arccosh\left( \frac{1+\sqrt{5}}{2} \right)}{n a} < \gamma_* .
		\end{equation}
		If $B(\gamma_*) < 2$, then the minimum is global. Moreover, we have $B'(\gamma) > 0$ for all $\gamma > \gamma_*$.
	\end{enumerate}
\end{proposition}
From the computations in the proofs we will see that the parameters $\gamma$ and $a$ are exchangeable.
\begin{proposition}
	The frame condition number of the Gabor system $\G(g_2, a \Z \times b \Z)$ with $(ab)^{-1} = n \in \N$ depends on $n$ and is given by
	\begin{equation}
		\frac{B}{A}= e^{2 a \gamma} \coth \left( \frac{\gamma n a}{2} \right)^2.
	\end{equation}
	The condition number is minimal if and only if
	\begin{equation}
		\gamma a = \frac{\log\left(n+\sqrt{n^2+4} \right) - \log(2)}{n}.
	\end{equation}
\end{proposition}
We see that optimality for the decay parameter $\gamma$ depends on the lattice parameter $a$, and vice versa, as well as on the (over)sampling rate $n$.

\subsection{One-sided exponential}
The one-sided exponential function is given by
\begin{equation}
	g(t) = \sqrt{2 \gamma} \, e^{- \gamma t} \indicator_{\R_+}(t), \; \gamma > 0.
\end{equation}
We note that it is known that the Gabor system $\G(g, a \Z \times b \Z)$ is a frame whenever $(ab)^{-1} \geq 1$ \cite[Sec.\ 4]{Jan96} (see also \cite{GroRomSto18}). Note that $g$ is not continuous (thus $g \notin M^1(\R)$), so having a frame at critical sampling rate is possible. Also, the one-sided exponential plays a central role when it comes to finding examples, other than the Gaussian, of zero-free ambiguity functions \cite{GroJamMal19}.
\begin{theorem}\label{thm:1-sided}
	Consider the Gabor system $\G(g, a \Z \times b \Z)$ with $(ab)^{-1} = n \in \N$. For fixed $n$, the frame condition number is minimal only for the lattice $(a_0/n) \Z \times (1/a_0) \Z$ with
	\begin{equation}
		a_0 = \frac{\arcsinh(n)}{\gamma}.
	\end{equation}
	The lower frame bound has a unique maximizing lattice $(a_+/n) \Z \times (1/a_+) \Z$ with
	\begin{equation}
		a_+ < \frac{\arcsinh(n)}{\gamma}.
	\end{equation}
	The upper frame bound has a unique minimizing lattice $(a^-/n) \Z \times (1/a^-) \Z$ with
	\begin{equation}
		a^- > \frac{\arcsinh(n)}{\gamma}.
	\end{equation}
\end{theorem}

\subsection{Two-sided exponential}
Lastly, consider the two-sided exponential, which is given by
\begin{equation}
	g(t) = \sqrt{\gamma} \, e^{-\gamma |t|}, \; \gamma > 0.
\end{equation}
The Gabor system $\G(g, a \Z \times b \Z)$ is a frame whenever $(a b)^{-1} > 1$ \cite[Sec.\ 5]{Janssen_CriticalDensity_2003}. Note that the case of critical sampling is excluded this time. This is a consequence of the Balian-Low theorem as $g \in M^1(\R)$. We also refer to \cite{Janssen_Tie_2003} where windows with certain convexity properties on $\R_+$ have been studied, and the two-sided exponential falls into this class.
\begin{theorem}\label{thm:2-sided}
	Consider the Gabor system $\G(g, a \Z \times b \Z)$ with $(ab)^{-1} = n \in \N$. For fixed $n$, the frame condition number is minimal only for the lattice $(a_0/n) \Z \times (1/a_0) \Z$ where $a_0$ is contained in some open interval depending on $n$:
	\begin{equation}
		a_0 \in \tfrac{1}{\gamma} \, I_n \subset \tfrac{1}{\gamma} \left( \tfrac{\eta_n}{2}, 2 \eta_n \right), \quad \eta_n = 2 \arccosh(n).
	\end{equation}
	The lower frame bound has a unique maximum, dependent on the oversampling rate $2 \leq n \in \N$ ($A=0$ for $n=1$). The maximizing lattice $(a_+/n) \Z \times (1/a_+) \Z$ satisfies $a_+ \in 1/\gamma \, I_n$.
	
	The upper frame bound has a unique minimum, dependent on the oversampling rate $n \in \N$. The minimizing lattice $(a^-/n) \Z \times (1/a^-) \Z$ satisfies $a^- \in 1/\gamma \, I_n$.
\end{theorem}
We show analytically that $a^- < a_0 < a_+$, and for $n \geq 4$ we have $a^- < 2 \arccosh(n)/\gamma < a_+$.

\section{Notation and preliminaries}\label{sec:notation}
We briefly clarify the setting. Our notation is more or less the same as in the textbook of Gr\"ochenig \cite{Gro01} and the reader familiar with Gabor systems and frames may simply skip this section. For two functions $f,g \in \Lt[]$, the inner product and induced norm are given by
\begin{equation}
	\langle f, g \rangle = \int_\R f(t) \, \overline{g(t)} \, dt
	\quad \text{ and } \quad
	\norm{f}_2^2 = \langle f, f \rangle.
\end{equation}
Here $\overline{g}$ is the complex conjugation of $g$. The Fourier transform of a (suitable) function $f$ is
\begin{equation}
	\F f (\omega) = \int_\R f(t) e^{-2 \pi i \omega t} \, dt.
\end{equation}
The operator $\F$ extends to a unitary operator on $\Lt[]$. After having settled the notation, we come to defining the central objects of this work. We consider Gabor systems for the Hilbert space $\Lt[]$ over rectangular lattices:
\begin{equation}
	\G(g, \L_{a, b}) = \{ \pi(\l) g \mid \l \in \L_{a,b}\},
	\quad \text{ with} \quad
	\L_{a, b} = a \Z \times b \Z.
\end{equation}
Here $\pi(\l)$ is a unitary operator on $\Lt[]$, usually called a time-frequency shift by $\l \in \R^2$;
\begin{equation}
	\pi(\l) g(t) = M_\omega T_x g(t) = g(t-x) e^{2 \pi i \omega t}, \qquad \l = (x,\omega).
\end{equation}
The operators $T_x$ and $M_\omega$ are the familiar translation and modulation operator, respectively. 
A Gabor system is a frame if and only if the frame inequality is fulfilled
\begin{equation}\label{eq:frame}
	A \norm{f}_{L^2}^2 \leq \sum_{\l \in \L_{a,b}} | \langle f, \pi(\l) g \rangle |^2 \leq B \norm{f}_{L^2}^2, \quad \forall f \in \Lt[],
\end{equation}
for some positive constants $0 < A \leq B < \infty$, which we call frame bounds. Note that these bounds depend crucially on $g$ as well as $\L_{a,b}$. The middle expression in \eqref{eq:frame} is derived by sampling the short-time Fourier transform $V_gf$ on the lattice $\L_{a,b}$:
\begin{equation}
	V_g f(x,\omega) = \langle f, \pi(\l) g \rangle = \int_\R f(t) \overline{g(t-x)} e^{-2 \pi i \omega t} \, dt.
\end{equation}
Being a frame usually requires some redundancy of the system. The number $(a b)^{-1}$ yields the density of the lattice, which is the average number of lattice points per unit area. In the context of Gabor systems, we also speak of the (over)sampling rate. Comparably to the Nyquist rate for band-limited functions, we need $(a b)^{-1} \geq 1$. The best achievable constants in \eqref{eq:frame} are the spectral bounds of the associated Gabor frame operator which is given by
\begin{equation}
	S_{g, \L_{a,b}} f = \sum_{\l \in \L_{a, b}} \langle f, \pi(\l) g \rangle \, \pi(\l) g.
\end{equation}
In case of a Gabor frame, any element in our Hilbert space has a stable expansion of the form
\begin{equation}
	f = \sum_{\l \in \L_{a,b}} c_\l \, \pi(\l) g,
\end{equation}
with $(c_\l) \in \ell^2(\L_{a, b})$ (see also \cite{Jan81}, \cite{LyuSei99}). Note that $(ab)^{-1} \geq 1$ is a necessary condition which may be far from sufficient for expansions of the above type.
 
We recall the notion of the frame set of a window function $g$ \cite{Gro14} (see also \cite{Faulhuber_LyuNes_2019}) for rectangular lattices. The rectangular (or reduced) frame set of $g \in \Lt[]$ is given by
\begin{equation}
	\mathfrak{F}_{r} (g) = \{ \L_{a, b} = a \Z \times b \Z \mid \G(g,\L_{a, b}) \textnormal{ is a frame}\}.
\end{equation}
Note that $\L_{a,b} \in \mathfrak{F}_r(g)$ does not necessarily imply $\L_{b,a} \in \mathfrak{F}_r(g)$ (counterexamples may easily be constructed by, e.g., using the box function), but for all $\L_{a,b}\subseteq\R^2$ and all $g\in \Lt[]$ holds
\begin{equation}
	\Lambda_{a,b}\in\mathfrak{F}_r(g)
	\quad \Longleftrightarrow \quad
	\Lambda_{b,a}\in\mathfrak{F}_r(\mathcal{F} g).
\end{equation}
This a special case of the \emph{symplectic invariance of  Gabor systems} \cite{Fol89, Gos11, Gro01}. We formulate another (simplified) version only for dilations, which is sufficient for our purposes.
\begin{proposition}\label{prop:symplectic_invariance}
Let $g\in \Lt[]$ be given. We denote by $D_\gamma$ ($\gamma > 0$) the unitary dilation $D_\gamma g(t):= \gamma^{-1/2}g(\gamma^{-1}t)$. Then for all $\L_{a,b}\subseteq\R^2$ holds
\begin{equation}
	\L_{a,b}\in\mathfrak{F}_r(g)
	\quad \Longleftrightarrow \quad
	\L_{\gamma a, \gamma^{-1}b}\in\mathfrak{F}_r(D_\gamma g).
\end{equation}
Furthermore, the Gabor systems $\G(g, \L_{a, b})$ and $\G(D_\gamma g, \L_{\gamma a,b/\gamma})$ possess the same frame bounds.
\end{proposition}

A necessary condition, mentioned already several times, on $\L_{a,b}$ to be contained in $\mathfrak{F}_{r}$ of any function in the Hilbert space $\Lt[]$ is given by the density theorem for Gabor systems (see also \cite{Hei07}) and can be summed up in the following way:
\begin{equation}
	\mathfrak{F}_{r} (g) \subset \{ \L_{a,b} \mid (ab)^{-1} \geq 1 \}.
\end{equation}
Assuming that $g \in S_0(\R) = \{ f \in \Lt[] \mid V_f f \in L^1(\R^2) \}$, which is known as Feichtinger's algebra \cite{Fei81}, \cite{Fei83} (see also \cite{BenOko_ModulationSpaces}, \cite{Jakobsen_S0_2018}), it is known that the frame set is open and contains a neighborhood of $\mathbf{0}$ \cite{Gro14}. This means that for any non-zero window from $S_0(\R)$ there exists a sufficiently high oversampling rate $(ab)^{-1}$, depending on the window, such that Gabor frames exist. 
The case of density 1 is called critical density and only at this level Gabor orthonormal bases exist. However, for windows in Feichtinger's algebra the Balian-Low theorem (\cite{Bal81}, \cite{Low85}, see also \cite[Chap.\ 8.3]{Gro01}) shows that there exist no Gabor frames at critical density.
For the hyperbolic secant, Janssen and Strohmer \cite{JanStr02} could show that the rectangular frame set is the largest possible;
\begin{equation}
	\mathfrak{F}_{r} (\sech) = \{ \L_{a, b} \mid (a b)^{-1} > 1\}.
\end{equation}
There are other functions for which we now know the rectangular frame set. Among these are totally positive functions of finite type (greater than or equal to 2) which all possess a full rectangular frame set \cite{GroechenigStoeckler_TotallyPositive_2013}. Note
that the two-sided exponential belongs to this class of functions and the hyperbolic secant is a totally positive function of infinite type, and both are also in $S_0(\R)$. Despite the fact that they both belong to the nice function space $S_0(\R)$, they exhibit very different optimality properties. The one-sided exponential is a totally positive function of (finite) type 1 and not in $S_0(\R)$. Its frame set is the largest possible, also containing all rectangular lattices at critical density. Another example is the indicator function of a finite interval and its frame set has an extremely complicated structure \cite{DaiSun_ABC_2015, DaiSun_ABC_2016}. The first and still only window for which a full characterization of Gabor frames has been known is the Gaussian window \cite{Lyu92}, \cite{Sei92_1}, \cite{SeiWal92}. For more general density results we refer to \cite{JakLem_Density_2016}. For more details on frames and Gabor systems we refer to the textbooks of Christensen \cite{Christensen_2016} and Gr\"ochenig \cite{Gro01} and for related results we also refer to the textbook of Folland \cite{Fol89}.

After having determined the frame set of a function $g$, a natural follow-up question is on the optimality of a lattice $\L_{a,b}\in \mathfrak{F}_{r}(g)$ of fixed density $n$. Several things can be meant by that. We could look for a maximizer of the lower bound $A(a,b)$, a minimizer of the upper bound $B(a,b)$ or a minimizer of the condition number $\kappa(a,b) = B(a,b)/A(a,b)$. Neither the existence nor the uniqueness of optimizing pairs $(a,b)$ is an obvious matter. Even in the uniqueness case, the optimal lattices will not necessarily be the same for all quantities.

There are several theoretical ways to investigate the behavior of the frame bounds. In the case of integer or rational density, a popular method is the Zak transform \cite{Jan88}.
Alternatively, one can turn to duality theory,
where the ground work is due to Daubechies, Landau and Landau \cite{DauLanLan95}, Janssen \cite{Jan95}, Ron and Shen \cite{RonShe97} and Wexler and Raz \cite{WexRaz90}. For a more thorough treatise of the topic we refer to \cite{GroKop19}. 

In \cite{Jan96}, Janssen computed the optimal frame bounds for 6 different windows. We built upon those computations and determine the critical points with respect to the lattice parameters. The case of the Gaussian and rectangular lattices has already been treated by Faulhuber and Steinerberger in \cite{FauSte17}, where optimality of the square lattice (in all three senses) under rectangular lattices of integer density has been shown. An optimality result for the upper bound and the hexagonal lattice (even density) followed \cite{Faulhuber_Hexagonal_2018} (see also \cite{Montgomery_Theta_1988}) and lastly B\'etermin, Faulhuber and Steinerberger \cite{BetFauSte21} proved the optimality of the lower bound for the hexagonal lattice (even density), which also implies the optimality of the hexagonal lattice for the condition number. In what follows, we will see that this does not seem to be a general pattern.

\section{Preparations}\label{sec:prep}
\subsection{The hyperbolic functions}\label{sec:hyperbolic} 

Hyperbolic functions will be omnipresent throughout the manuscript. We will only consider them restricted to the domain $[0,\infty)$. 
The monotonicity properties of hyperbolic functions can be read easily from their Taylor expansions or the relation to the exponential function. Truncations of the Taylor series also provide good upper and lower bounds.
We mention these rather simple facts because the proofs which follow require evaluating and bounding these functions multiple times at various points. To verify the
estimates, it would often suffice to evaluate the corresponding tenth Taylor polynomial. 
Of course, the reader is also welcome to use their preferred software. For numerical checks, we used Mathematica \cite{Mathematica}, which evaluates these functions to arbitrary precision (we used a precision to the $24^{\text{th}}$ digit)\footnote{\ The Mathematica notebook is provided for download as an ancillary file.}. Still, our arguments are elementary and can also be checked by hand, but the bookkeeping is delicate at several points and makes the proofs non-trivial and rather technical.

Besides the various elementary identities involving the derivatives or squares of the hyperbolic functions\footnote{\ Some identities and explanations are provided in the supplement}, which can all be found in \cite{GradshteynRyzhik2014}, we will also use the following equation \cite[eq.\ 86]{Bruckman1977}
\begin{equation}\label{HyperTransform}
x \sum\limits_{k=-\infty}^\infty \sech(k x)^2 = 2 + \frac{2}{x} \sum\limits_{k=0}^\infty \csch\left(\left(k+\frac{1}{2}\right)\,\frac{1}{x}\right)^2.
\end{equation}
This and many other equations in \cite{Bruckman1977} illuminate the deeper connection of the hyperbolic functions to elliptic functions and elliptic integrals, which also appear to be closely linked to Gabor frames with the Gaussian window (see \cite{Jan96}, compare also \cite{Faulhuber_Rama_2019}).

\subsection{Interval estimate}\label{sec:intervals}
We will often speak of interval estimates. By that we mean bounding an expression consisting of monotonic functions (or otherwise well-known bounded functions) on an interval $[x_0, y_0]$ by the worst possible value. In some of our calculations we will use the symbols $\nearrow$ and $\searrow$ to denote that an expression is increasing or decreasing, respectively. For example, $\tanh$ is strictly monotonically increasing and $\csch$ is strictly monotonically decreasing, both strictly positive. Then $\csch(t)-\csch(2t)>0$ and
\begin{equation}
	\underbrace{\tanh(t)}_{\nearrow}\big(\underbrace{\csch(t)}_{\searrow}\underbrace{-\csch(2t)}_{\nearrow}\big) < \tanh(y_0) \big(\csch(x_0)-\csch(2y_0)\big), \quad t \in (x_0, y_0).
\end{equation}
which gives us as a rough upper bound as in Figure \ref{fig:intervals} (if the end point is $0$ or $\infty$, we would take the limit, presuming it exists). The arrows indicate the monotonicity properties we applied for the estimate. Of course, the sign of each factor also goes into the overall estimate.

\begin{figure}[ht]
	\includegraphics[width=.75\textwidth]{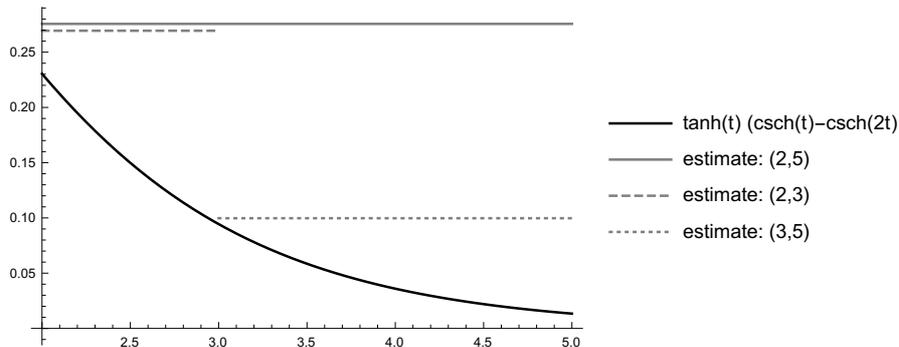}
	\caption{Estimate of the above type for $\tanh(t) (\csch(t)-\csch(2t))$ on the interval $(2,5)$. In order to get more appropriate estimates we may use the sub-intervals $(2,3)$ and $(3,5)$.}\label{fig:intervals} 
\end{figure}

\section*{Proofs}
\section{The hyperbolic secant}\label{sec:sech}
In this section we will prove Theorem \ref{thm:sech}. 
By Proposition \ref{prop:symplectic_invariance}, we only need to prove it for the case $\gamma =1$ in order to obtain Corollary \ref{cor:sech}. We will show that
the lower frame bound, the upper frame bound and the condition number of the Gabor system $\G(\sqrt{\pi/2} \, \sech(\pi t), a \Z \times b \Z)$ are optimal if and only if $a=b$, assuming $2 \leq (ab)^{-1} \in \N$. We begin with a transformation of the problem.
We use the formula provided by Janssen \cite[eq.\ (7.8)]{Jan96} (with the additional normalizing factor $\frac{\pi}{2}$) to express the frame bounds. To simplify the expressions, we set
\begin{align}
	f_A(t) = t \sum\limits_{k=0}^\infty \sech(\alpha_k t)^2
	\quad \text{ and } \quad
	f_B(t) = t \sum\limits_{k=0}^\infty \csch(\alpha_k t)^2, \quad t\in (0,\infty), 
\end{align}
where $\alpha_k = \pi(k+1/2)$. 
We use a change of variables which respects the fact that $(ab)^{-1}=n$:
\begin{equation}\label{eq:change_variables}
	(a,b) \mapsto (\tfrac{\eta}{n}, \tfrac{1}{\eta}).
\end{equation}
The exact frame bounds can now be expressed by (compare \cite[eq.\ (7.8)]{Jan96})
\begin{align}
	A(\tfrac{\eta}{n}, \tfrac{1}{\eta}) &= \tfrac{n\pi}{2}	\, 
\sum\limits_{k=-\infty}^\infty \tfrac{\eta}{n} \, \sech(\pi (k+\tfrac{1}{2})\, \tfrac{\eta}{n})^2  + \tfrac{n\pi}{2} \, \sum\limits_{k=-\infty}^\infty \tfrac{1}{\eta} \, \sech(\pi (k+\tfrac{1}{2})\, \tfrac{1}{\eta} )^2  - n \\[0.5ex]
	&= n\pi\, f_A(\tfrac{\eta}{n}) + n\pi \, f_A(\tfrac{1}{\eta}) - n, \\[0.5ex]			
	B(\tfrac{\eta}{n}, \tfrac{1}{\eta}) &= \tfrac{n\pi}{2}	\, 
\sum\limits_{k=-\infty}^\infty \tfrac{\eta}{n} \, \sech(\pi k \tfrac{\eta}{n})^2  + \tfrac{n\pi}{2} \, \sum\limits_{k=-\infty}^\infty \tfrac{1}{\eta} \, \sech(\pi k \tfrac{1}{\eta} )^2  - n \\[0.5ex]
	&= n\pi\, f_B(\tfrac{n}{\eta}) + n\pi \, f_B(\eta) - n
\end{align}
\begin{figure}[h!t]
	\subfigure[Lower bound for different densities.]{
	    \includegraphics[width=.45\textwidth]{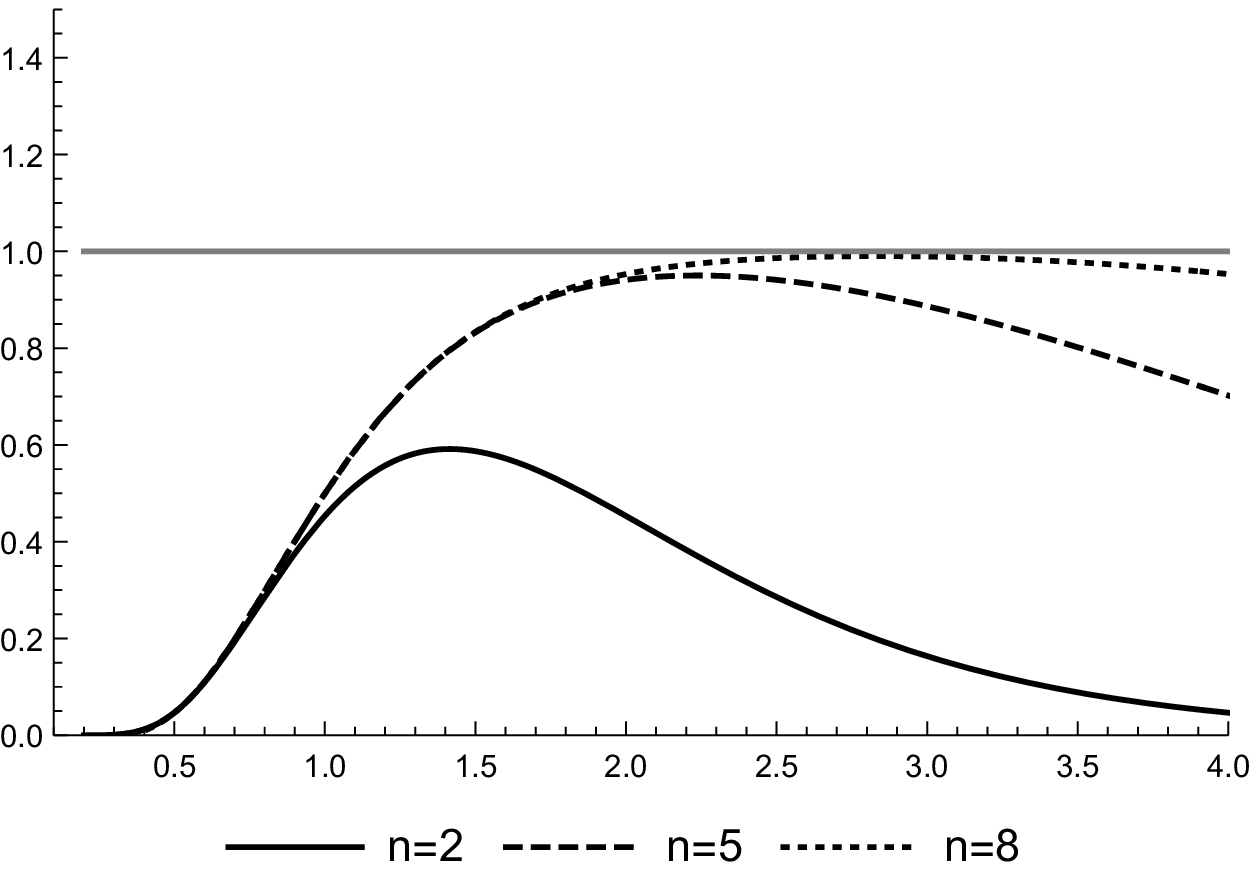}
	}
    \hfill
    \subfigure[Upper bound for different densities.]{
	    \includegraphics[width=.45\textwidth]{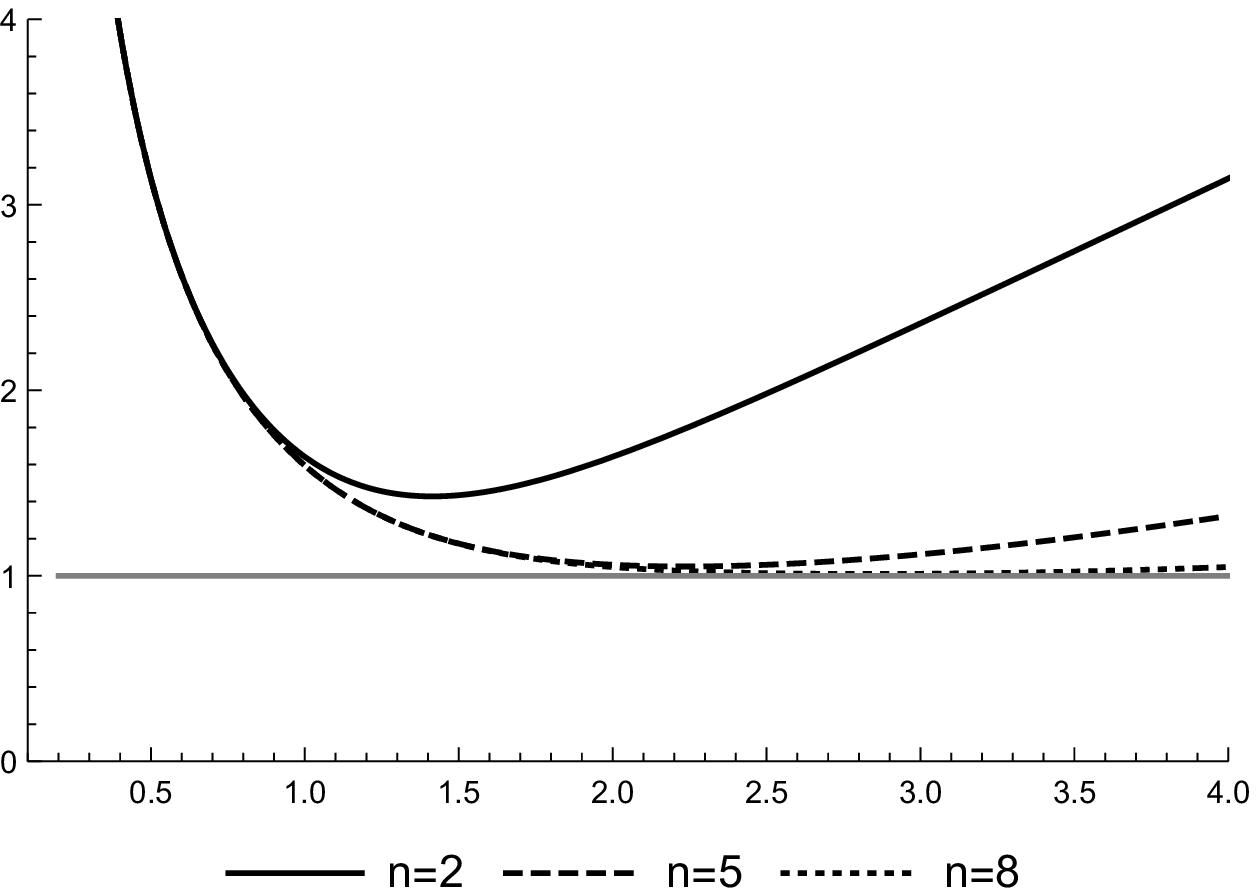}
	}
    \caption{The re-normalized optimal bounds $n^{-1} A(\eta/n,1/n)$ and $n^{-1} B(\eta/n,1/n)$ for densities $n \in \{2,5,8\}$. As $n$ grows, the bounds become extremely flat around the optimum. For both cases the optimizing parameter $\eta$ (depending on $n$) yields a square lattice of density $n$.}
    \label{fig:sech_A_B}
\end{figure}
where we used the symmetry of $\sech$ and Identity \eqref{HyperTransform}. 

The series defining $f_A$ and $f_B$ converge locally uniformly. By first taking formal derivatives (taking term-wise derivatives within the series) we see that also these series converge locally uniformly for all formal derivatives. Hence, $f_A$ and $f_B$ are infinitely continuously differentiable and we are allowed to simply differentiate term-wise (see \cite[Chap.~7]{Rud76} for technical details). The quantity $\eta$ describes the lattice geometry and we have a square lattice if $\eta = n^{1/2}$. We note that optimizing the bounds gets more and more difficult as the density grows since the curves get flatter and flatter as indicated in Figure \ref{fig:sech_A_B}.

Our strategy is the following: From the algebraic structure of the problem, it follows immediately that $\eta= n^{1/2}$ gives a critical point, which by \eqref{eq:change_variables} shows that the square lattice is critical in the set of rectangular lattices. Then, we will prove certain monotonicity properties of the involved terms and their derivatives to exclude further critical points.

Determining the critical points is achieved by finding all values $\eta \in \R_+$ such that 
\begin{align}\label{eq:extremal_sech}
	\begin{split}
		\tfrac{\eta}{n\pi} \, \tfrac{d}{d \eta} A \left( \tfrac{\eta}{n}, \tfrac{1}{\eta} \right)
		&= \tfrac{\eta}{n} \, f_A' ( \tfrac{\eta}{n} ) - \tfrac{1}{\eta} \, f_A' ( \tfrac{1}{\eta}) 
		= h_A ( \tfrac{\eta}{n} ) - h_A ( \tfrac{1}{\eta} ) = 0, \\[1ex]
	 	\tfrac{\eta}{n\pi} \, \tfrac{d}{d \eta} B \left( \tfrac{\eta}{n}, \tfrac{1}{\eta} \right)
	 	&= -\tfrac{n}{\eta} \, f_B' ( \tfrac{n}{\eta} ) + \eta \, f_B' ( \eta )
		= -h_B ( \tfrac{n}{\eta} ) + h_B(\eta) = 0.
	\end{split}
\end{align}

\begin{wrapfigure}{r}{0.45\textwidth}
	\vspace*{-0.5cm}
	\includegraphics[width=.425\textwidth]{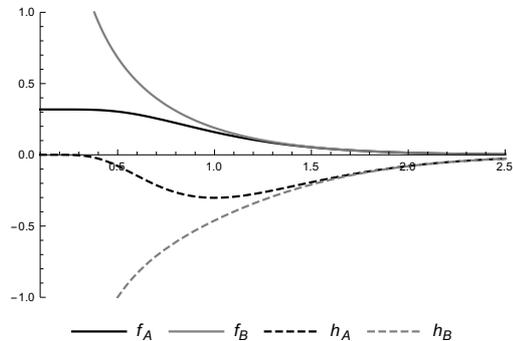}
	\vspace*{-0.3cm}
	\caption{Comparison of the behavior of $f_A$ and $f_B$ as well as $h_A$ and $h_B$.}\label{fig:sech_f_h}
\end{wrapfigure}
Here we used the abbreviation
\begin{equation}
	\hspace*{-.5\textwidth}
	h_A(t)=t \, f_A'(t)
	\quad \text{ and } \quad
	h_B(t)=t \, f_B'(t).
\end{equation}
Obviously, the expressions in \eqref{eq:extremal_sech} vanish for $\eta = n^{1/2}$, which corresponds to the square lattice of density $n$. We seek to show that there are no other critical points to prove Theorem \ref{thm:sech}.
So far, the lower and upper bound's dependence on $f_*$ and $h_*$, as well as the structure of $f_*$ and $h_*$ seem to be nearly the same. Yet, the accumulation of derivatives of $\sech$ and $\csch$ secretly amplifies several differences (see also Figure \ref{fig:sech_f_h}).

\medskip

\subsection{The lower bound}
In \cite{Bruckman1977}, we find the identity (abbreviated with our notation above)
\begin{equation}\label{eq:Bruckmann}
	\pi f_A(t) +\pi f_A(t^{-1}) = 1,
\end{equation}
which implies the point symmetry $h_A(t) = h_A(t^{-1})$. Rather remarkably, we could have concluded \eqref{eq:Bruckmann} with only our knowledge about Gabor frames at critical density (a similar remark is given in \cite{Jan96}): $A(\eta,1/\eta) = \pi f_A(\eta) + \pi f_a (\eta^{-1}) - 1 = 0$ because of the Balian-Low theorem \cite{Bal81}, \cite{Low85} (see also \cite[Chap.~8.4]{Gro01}). We will now prove that the following equivalence.
\begin{lemma}
	For all $x,y\in (0,\infty)$ holds
	\begin{equation}\label{eq:equivalence_A}
		h_A(x)=h_A(y)
		\qquad \Longleftrightarrow \qquad
		x=y \; \vee \; x=y^{-1}.
	\end{equation}
\end{lemma}
\begin{proof}
One implication is rather simple: the case $x=y$ is trivial and the case $x=y^{-1}$ follows from the already established point symmetry of $h_A$. 
We start the proof of the other implication with the observation that due to the symmetry, it suffices to prove the strict monotonicity on $(1,\infty)$ to exclude further solutions for \eqref{eq:equivalence_A}. We split our computations in two parts. The monotonicity is straightforward on the interval $(1.03, \infty)$. On the part $(1,1.04)$, we prove that $h_A$ is strictly convex, i.e., $h_A'$ is strictly increasing. Taking into account that $h_A'(1) =0$, due to the symmetry, $h'_A(t)>0$ for all $t>1$.

\bigskip

\textbf{Monotonicity away from 1}: 
Firstly, we compute the derivative\footnote{\textsuperscript{, \scriptsize{4}} \suppfoot \ Computation 3 is used twice.} of $h_a(t) = t \, f_A'(t)$.
\begin{align}
	h_A'(t)& = \sum\limits_{k=0}^\infty \sech(\alpha_k t)^2 \, \left(1-2\alpha_k^2 t^2 \, \sech(\alpha_k t)^2+4\alpha_k^2 t^2 \, \tanh(\alpha_k t)^2-6\alpha_k t \, \tanh(\alpha_k t)\right).
\end{align}
We split the expression in the brackets in two. The goal is to determine points which satisfy
\begin{equation}
	1-2x^2\,\sech(x)^2>0
	\quad \text{ and }\quad
	4x^2 \, \tanh(x)^2-6x \, \tanh(x) \geq 0.
\end{equation}
The second inequality is fulfilled as long as $x\,\tanh(x)\geq\frac{3}{2}$, which is a strictly monotonically increasing function in $x>0$. By evaluating, one can check that the solution of $x\,\tanh(x)=\frac{3}{2}$ lies in $(1.6218,1.6219)$. For the former inequality, we observe
\begin{align}
	\lim_{x \to 0} 1-2x^2\,\sech(x)^2 & = \lim_{x \to \infty} 1-2x^2\,\sech(x)^2=1, \\
	\tfrac{d}{d x} \left[1-2x^2\,\sech(x)^2 \right] & = 4x\sech(x)^2(x\tanh(x)-1).
\end{align}
Since the fixed point of $\coth$ lies in $(1.19, 1.2)$, this yields the only critical point of $x \mapsto 1-2x^2 \sech(x)^2$ for $x > 0$. This must yield the minimum. Therefore, it holds that
\begin{align}
	1-2x^2\,\sech(x)^2 > 1 - 2 \cdot 1.2^2 \cdot \sech(1.19)^2 > 0.1 > 0, \quad \forall x > 0.
\end{align}
As $\alpha_k t\geq \alpha_0 t$, for all $t\geq 1.6219\,\alpha_0^{-1} = \frac{3.2438}{\pi}\in(1.03,1.04)$, we see that $h_A'(t)>0$ for $t > 1.03$.

\bigskip

\textbf{Strict convexity close to 1}:
Since $h_A'(1)=0$, trying to find positive lower bounds for $h_A'$ close to 1 is playing a losing game, so we turn to the second derivative\footnotemark[3] of $h_A$.
\begin{align}
	h_A''(t)&= 2\,\sum\limits_{k=0}^\infty\alpha_k\sech(\alpha_k t)^2 \times
	\Big(8\alpha_k^2t^2\sech(\alpha_k t)^2\tanh(\alpha_k t) -5\alpha_kt\sech(\alpha_k t)^2\\
	& \qquad  \qquad -4\tanh(\alpha_k t)+10\alpha_k t\tanh(\alpha_k t)^2 -4\alpha_k^2t^2\tanh(\alpha_k t)^3\Big).
\end{align}
We do a similar separation as above.

\textbf{Part 1:} Recall that $\alpha_k \geq \tfrac{\pi}{2}$ for $k \geq 0$. For all $t\in[1, 1.04]$, holds (due to $x \tanh(x)\nearrow$)
\begin{align}
	& \,\quad 8\alpha_k^2t^2\sech(\alpha_k t)^2\tanh(\alpha_k t) -5\alpha_kt\sech(\alpha_k t)^2 	\\
	&\geq \, \sech(\alpha_k t)^2(8 \, \tfrac{\pi}{2} \tanh(\tfrac{\pi}{2})-5)>6 \sech(\alpha_k t)^2 >0.
\end{align}

\textbf{Part 2:} We compute
\begin{align}
	&\quad \, \alpha_k\sech(\alpha_k t)^2 \, \Big( -4\tanh(\alpha_k t)+10\alpha_k t\tanh(\alpha_k t)^2 -4\alpha_k^2t^2\tanh(\alpha_k t)^3\Big)\\
	&= \, 2 \alpha_k\sech(\alpha_k t)^2 \tanh(\alpha_k t) \, p\big(\alpha_k t\tanh(\alpha_k t)\big),
\end{align}
where $p(x) = -2x^2+5x-2 = -2(x-\frac{1}{2})(x-2)$.
First of all, $p(x) > -3x^2$ on $(1,\infty)$.
By examining the graph of $p$ and taking into account $t\in (1, 1.04)$, one can verify\footnotemark 
\begin{align}
	p(\alpha_0 t) \tanh(\alpha_0 t) &>0, \\
	p(\alpha_k t) \tanh(\alpha_k t) &<0, \qquad k\in\N.
\end{align}
We estimate
\begin{align}
	0 \geq 2 \alpha_k \tanh(\alpha_k t)\sech(\alpha_k t) \, p \big(\alpha_k t\tanh(\alpha_k t)\big) 
	> -12\big(\alpha_k t\big)^3 e^{-\alpha_k t} >-16.14,
\end{align}
where the last estimate is due to the fact that $\abs{\tanh(x)}<1$, $e^{-x}\leq\sech(x)$ and the maximum of  $x^3e^{-x}$ is attained at $x=3$.
To sum up, for all $t\in(1,1.04)$ we have\footnote{\textsuperscript{, \scriptsize{6}} \suppfoot} 
\begin{align}
	\frac{1}{2}h_a''(t) 
	& > 3 \cdot 0.07 -16.14 \cdot \tfrac{0.04}{2.46}+0.19
	> 0.13 >0.
\end{align}
All in all, we conclude that $h_a$ is strictly monotonically increasing on $(1,\infty)$ and due to the symmetry it has to be strictly monotonically decreasing on $(0,1)$. Let now $x,y \in (0,\infty)$ be two distinct points satisfying $h_A(x)=h_A(y)$. Without loss of generality, we can assume $x<y$. By the established monotonicity, $x<1<y$, implying also $1<x^{-1}$. So, we can write
\begin{equation}
h_A(\tfrac{1}{x})=h_A(x)=h_A(y).
\end{equation}
The monotonicity on $(1,\infty)$ now implies $x^{-1} = y$.
\end{proof}
So, by \eqref{eq:equivalence_A} of the above lemma we have
\begin{equation}
	h_A(\tfrac{\eta}{n}) = h_A(\tfrac{1}{\eta})
	\qquad {\Longleftrightarrow} \qquad
	\frac{\eta}{ n} = \frac{1}{\eta} \; \vee\; \frac{\eta}{ n} = \eta
	\qquad \Longleftrightarrow\qquad \eta = \sqrt{n} \; \vee \;  n=1.
\end{equation}
The case of critical density ($n=1$) is trivial and $A$ vanishes identically in this case. For all other densities $n \in \N$, $\eta = \sqrt{ n}$ is the only critical point. To conclude the statement on the lower frame bound, recall that for $(ab)^{-1} =n$, we used the map $(a,b) \mapsto (\eta/n, 1/\eta)$. Hence, the square lattice $n^{-1/2} \Z \times n^{-1/2} \Z$ yields the unique maximum of $A$ for density $2 \leq n \in \N$.

\subsection{The upper bound}
Contrary to $h_A$, the function $h_B$ possesses no symmetries, but is strictly monotonically increasing. We will split $(0,\infty)$ in multiple intervals. There is no denial that there are far better estimates than the following ones. In fact, numerical inspection of the function $\psi$ suggest $\psi>1$, while we only show its strict positivity. However, the estimates here are very simple and easy to verify. Trying to improve them would likely make the conditions under which the interval splitting is obtained significantly more complicated. At least two cases are probably unavoidable, since the function under consideration $\psi$ behaves rather differently for very small and for very big arguments. Our main strategy is similar.

\begin{lemma}\label{lemma:h_b_monotonic}
	For all $x,y\in (0,\infty)$ holds
	\begin{equation}\label{eq:equivalence_B}
		h_B(x) = h_B(y) \quad \Longleftrightarrow \quad x = y.
	\end{equation}
\end{lemma}
\begin{proof}
We show that $h_B'(t) > 0$ for all $t \in (0, \infty)$, which then gives the proof.

We begin with computing the first derivative\footnotemark: 
\begin{align}
	h_B'(t)	= \sum\limits_{k=0}^\infty \csch(\alpha_k t)^2 \psi(\alpha_k t),
\end{align}
where $\psi(x) = 1-2x^2+6x^2 \coth(x)^2-6x\coth(x)$. If we prove that $\psi$ is strictly positive on $(0,\infty)$, then so is $h_B'$.

\textbf{Part 1: $x>\log (\sqrt{21})$.} 
Due to the monotonicity of $\exp$, one can easily convince themselves that $1.5 < \log(\sqrt{21})<1.6$.
Since $\coth$ is strictly monotonically decreasing, and satisfies $\coth(x) = 1+2/(\exp(2 x)-1)$, we estimate as follows for $x>\log(\sqrt{21})$:
\begin{align}
	\psi(x)> 1-2x^2+6x^2\cdot 1^2-6x\cdot 1+ \frac{2}{21-1} 
	 > 4\cdot(1.5-0.825)^2-1.7225 = 0.1 > 0.
\end{align}
\textbf{Part 2: }Close to zero, $\coth(x)$ is unbounded, so we have to take into account the interaction with $x$.
We rewrite $\psi$ on $(0,1.6)$ as 
\begin{equation}
	\psi(x) = q_1(x) + 6 \, q_2(x \, \coth(x)),
	\qquad q_1(x) = 1-2x^2
	\quad \text{ and } \quad
	q_2(x) = x(x-1).
\end{equation}	
Clearly, $q_2$ is strictly monotonically decreasing on $(0,\infty)$, whereas $q_1$ is strictly monotonically increasing and strictly positive on $(1,\infty)$. To that, $x\coth(x)$ is strictly monotonically increasing. The interval estimate of $\psi$ on an interval $(x_0,y_0)\subseteq (0,1.6)$ will satisfy
\begin{equation}
	\psi(x)>q_1(y_0)+6 \, q_2(x_0 \coth(x_0))>0,
\end{equation}
if we choose $x_0$, $y_0$ well. One can convince themselves of the inequality\footnote{\ \suppfoot} with interval estimates on $(0,\sqrt{0.5})$, $(0.7,1.01)$, $(1, 1.3)$ and $(1.29, 1.6)$ (as described in Subsection \ref{sec:intervals}). 
All in all, 
$h_B'$ is strictly positive on $(0,\infty)$.
\end{proof}
 We now consider what Lemma \ref{lemma:h_b_monotonic} means for the optimality of the upper bound $B$. We have established
\begin{equation}
	h_B(\tfrac{\eta}{n}) = h_B(\tfrac{1}{\eta})
	\quad \Longleftrightarrow \quad
	\frac{\eta}{ n} = \frac{1}{\eta}
	\quad \Longleftrightarrow \quad
	\eta = \sqrt{ n}.
\end{equation}
Furthermore, for all $\eta>\sqrt{ n}$ we have
\begin{equation}
	\tfrac{d}{d \eta}\ B(\tfrac{\eta}{n},\tfrac{1}{\eta}) > 0.
\end{equation}
Due to the symmetry about $\sqrt{n}$, the square lattice $n^{-1/2} \Z \times n^{-1/2} \Z$ is the unique minimizing lattice of $B$ among rectangular lattices of density $n \in \N$.

As $A$ is maximized by the square lattice $n^{-1/2} \Z \times n^{-1/2} \Z$ and $B$ is minimized by the square lattice $n^{-1/2} \Z \times n^{-1/2} \Z$, also the condition number $\kappa = B/A$ has to be minimal in this case.

In contrast to the Gaussian case, we do not know sufficient density conditions for the hyperbolic secant and general lattices to yield Gabor frames. So, before speculating about an optimality result for general lattices one should work out the lattice frame set of the hyperbolic secant. Nonetheless, we note that the conjecture of Strohmer and Beaver \cite{StrBea03} was based on the radial symmetry of the Gaussian in the time-frequency plane and connections to sphere packings (see also \cite{BetFauSte21}). By the characterization of radially symmetric ambiguity functions, which can only come from Gaussians and Hermite functions \cite{Fol89}, it is clear that the hyperbolic secant does not possess a radial symmetry in the time-frequency plane.

\section{Cut-off exponentials}\label{sec:cut}
\subsection{Support 1/\texorpdfstring{$b$}{b}} 
In this section we study frame bounds of cut-off exponentials with support in $[0, 1/b]$ (or any interval of length $1/b$) for the lattice $\L_{a,b} = a \Z \times b \Z$ with density $(a b)^{-1} \geq 1$.
We denote the cut-off exponential with decay parameter $\gamma > 0$ by
\begin{equation}
	g_{b,\gamma}(t) = C_{b,\gamma} \, e^{-\gamma t} \, \indicator_{[0,1/b]}(t),
\end{equation}
where $C_{b,\gamma}$ is a normalizing factor, which is explicitly given by $C_{b,\gamma} = \sqrt{2\gamma/(1-e^{-2 \gamma/b})}$. The bounds have been computed by Janssen \cite{Jan96} without the factor $C_{b,\gamma}$, so we only need a small adjustment compared to his result. We also refer to \cite[Sec.\ 3.4.4]{Daubechies_Lectures} where similar computations have been carried out. 
After proper re-normalization we end up with
\begin{equation}\label{eq:bounds_1_b}
	A = n \, \frac{2 \gamma a}{1 - e^{-2 \gamma a}} \, e^{- 2 \gamma a}
	\quad \text{ and } \quad
	B = n \, \frac{2 \gamma a}{1 - e^{-2 \gamma a}}, \qquad \gamma > 0.
\end{equation}
We may include the case of the box function ($\gamma = 0$) by a limiting procedure and by l'Hôspital's rule we obtain $A=B=n$. As we always fix $(ab)^{-1}$, we usually have one free parameter to optimize the frame bounds. This time we fix both lattice parameters, $a$ and $b$, and consider a family of windows parametrized by $\gamma > 0$.

\subsubsection{The lower frame bound}
We fix the oversampling rate $n$ and consider the lower frame bound $A$ as a function of $\gamma$. Taking the derivative yields the expression
\begin{equation}
	A'(\gamma) = -2 a n \, \frac{e^{2 a \gamma } (2 a \gamma -1)+1}{\left(e^{2 a \gamma }-1\right)^2} < 0.
\end{equation}
Verifying that the above expression is indeed negative is a simple exercise: to show that the numerator is positive, just observe that, for $x \in \R \backslash\{0\}$, $e^x > (1+x)$, which is its linearization at 0. Thus, numerator and denominator are positive and there is a minus sign in front.

\subsubsection{The upper frame bound}
With a very similar observation we also obtain that
\begin{equation}
	B'(\gamma) = 2 a n \, \frac{e^{2 a \gamma } \left(e^{2 a \gamma }-1-2 a \gamma\right)}{\left(e^{2 a \gamma}-1\right)^2} > 0.
\end{equation}

\begin{figure}[ht]
	\subfigure[Frame bounds for $\Z \times 1/n \, \Z$.]{
		\includegraphics[width=.45\textwidth]{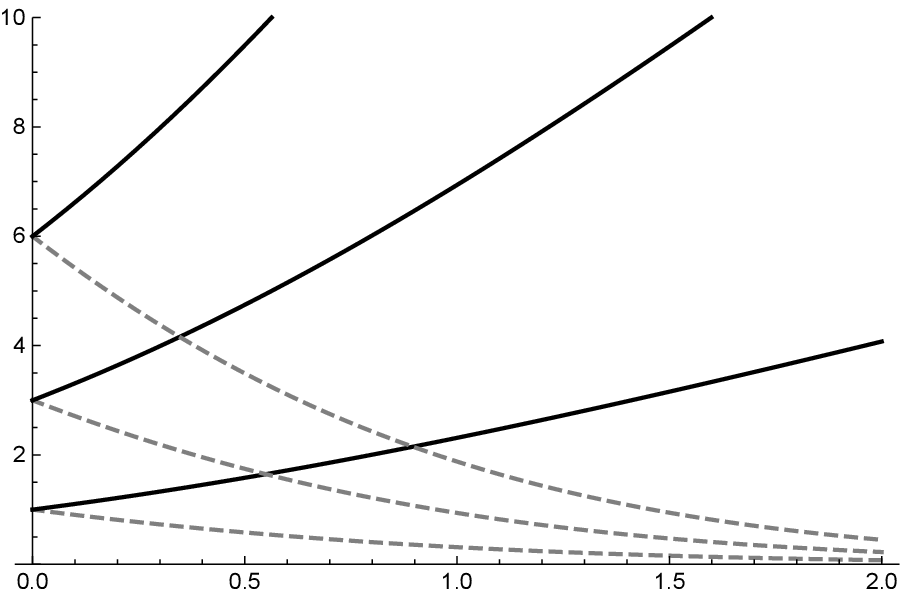}
	}
	\hfill
	\subfigure[Frame bounds for $1/\sqrt{n} \, \Z \times 1/\sqrt{n} \, \Z$.]{
		\includegraphics[width=.45\textwidth]{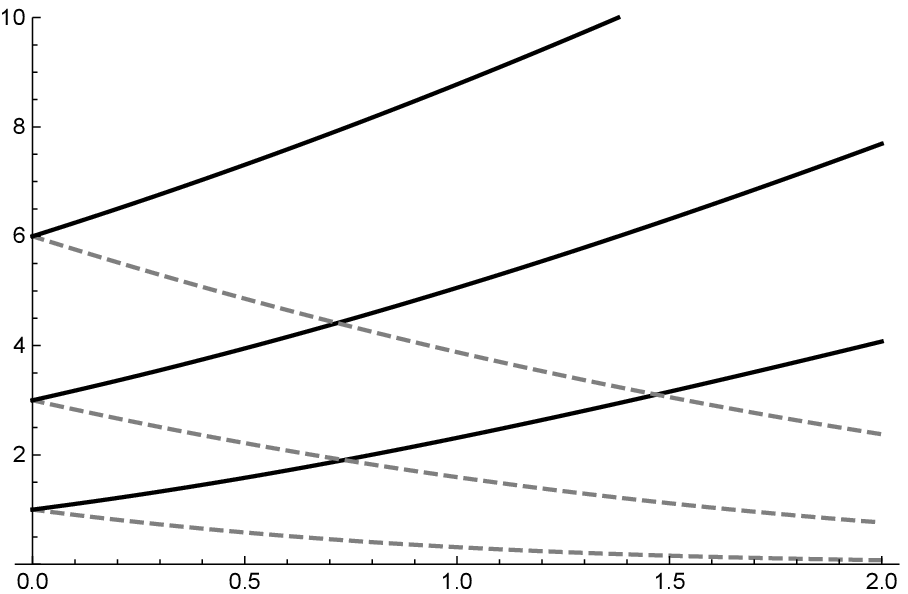}
	}
	\caption{Frame bounds of the Gabor systems $\G(g_\gamma, a \Z \times b \Z)$ for $(ab)^{-1} = n \in \{1,3,6\}$ and $a=1, \, b = 1/n$ and $a = b = \sqrt{1/n}$. For $\gamma = 0$, we always have a tight frame with bounds $A=B=n$. The upper frame bounds (black) are increasing with $\gamma$ whereas the lower frame bounds (gray, dashed) are decreasing.}\label{fig:bounds_1_b}
\end{figure}
From Figure \ref{fig:bounds_1_b} one may get the impression that the bounds deteriorate slower for the square lattice than for rectangular lattices. This is, however, only the case as long as $a > b$, which we deduce from \eqref{eq:bounds_1_b}.

\subsubsection{The condition number}
Finally, we easily observe that the condition number is
\begin{equation}
	\frac{B}{A} = e^{2 a \gamma}.
\end{equation}
We give some concluding remarks for the parameters of the family of the window $g_{b,\gamma}$. Fixing any rectangular lattice of density $n$, we see that the box function ($\gamma = 0$) is actually optimal within the family of localization windows $g_{b,\gamma}$. Thus, the additional decay without additional smoothness in the time variable yields to even worse behavior in the frequency variable. Therefore, the frame condition is only made worse by increasing the decay parameter $\gamma$.

From \eqref{eq:bounds_1_b} it is obvious that the decay parameter $\gamma$ and the lattice parameter $a$ are exchangeable for our computations. Note, however, that we change the window and the lattice at the same time. For any fixed $\gamma$ the lower bound decreases and the upper bound increases as $a$ grows. For $a \to 0$ we obtain a tight frame. This may be interpreted in a distributional sense: For $b \to \infty$ ($\Leftrightarrow \, a \to 0$), the function $g_{b,\gamma}$ tends to (a multiple of) the Dirac delta. Hence, even though we use an infinitely wide step size in the frequency domain, we know the function at every instance in time. The frame operator is therefore (a multiple of) the identity operator. In particular, an optimal lattice in terms of the frame bounds (or the frame condition number) does not exist.

\subsection{Support 2/\texorpdfstring{$b$}{b}}
Similarly to the previous case, we study frame bounds of cut-off exponentials, but this time with support in $[0, 2/b]$ for the lattice $\L_{a, b} = a \Z \times b \Z$ of density $(a b)^{-1} \geq 1$. We use (more or less) the same notation as in the previous section and denote the window by 
\begin{equation}
	g_{b/2,\gamma}(t) = C_{b/2,\gamma} \, e^{-\gamma t} \, \indicator_{[0,2/b]}(t),
\end{equation}
where $C_{b/2,\gamma}$ is the normalizing factor, which is explicitly given by $C_{b/2,\gamma} = \sqrt{2\gamma/(1-e^{-4\gamma/b})}$.
We skip the computational details again and refer to Janssen \cite[Sec.~3]{Jan96} where the bounds have been computed explicitly. With the proper normalization, the frame bounds are then
\begin{equation}
	A= C_{b/2,\gamma}^2 \, \frac{a (1-e^{-2\gamma/b})(1-e^{-\gamma/b})^2}{1-e^{-2 \gamma a}} \, e^{-2 a \gamma}
	\quad \text{ and } \quad
	B=  C_{b/2,\gamma}^2 \, \frac{a (1-e^{-2\gamma/b})(1+e^{-\gamma/b})^2}{1-e^{-2 \gamma a}}.
\end{equation}
By using the explicit expression for $C_{b/2,\gamma}$ and the fact that $(a b)^{-1} = n \ (\in \N)$, we get
\begin{equation}
	A = \frac{2 \gamma a (1-e^{-2\gamma n a})(1-e^{-\gamma n a})^2}{(1-e^{-2 \gamma a})(1-e^{-4\gamma n a})} \, e^{-2 a \gamma}
	\quad \text{ and } \quad
	B= \frac{2 \gamma a (1-e^{-2\gamma n a})(1+e^{-\gamma n a})^2}{(1-e^{-2 \gamma a})(1-e^{-4\gamma n a})}.
\end{equation}
For $\gamma \to 0$, the bounds' behavior is $A \to 0$ and $B \to 2$. In the other direction, i.e., for $\gamma \to \infty$ we have $A \to 0$ and $B \to \infty$. So, in both limiting cases we do not have a frame. This stands in contrast to the previous case, where we had a tight frame (orthonormal basis) with $A=B=n$ for the integer lattice $n^{-1/2} \Z \times n^{-1/2} \Z$ and $\gamma = 0$ (the box function).

\subsubsection{The lower bound}
Our next goal is to determine optimizing parameters for the lower bound. We start with re-writing the expression as
\begin{equation}
	A= \underbrace{\gamma a \left(\coth(\gamma a) - 1\right)}_{f(\gamma)} \, \underbrace{\left( 1 - \sech(\gamma n a) \right)}_{h(\gamma)}.
\end{equation}
Determining critical points with respect to $\gamma$ is equivalent to finding points where the sum of the logarithmic derivatives of $f$ and $h$ vanishes (because $f$ and $h$ are positive) or where:
\begin{equation}\label{eq:logA_2_b}
	\frac{f'(\gamma)}{f(\gamma)} = - \frac{h'(\gamma)}{h(\gamma)}.
\end{equation}
We compute
\begin{equation}
	\frac{f'(\gamma)}{f(\gamma)} = -a + \frac{1}{\gamma} - a \coth(\gamma a) = -a + \frac{1}{\gamma}\underbrace{(1- \gamma a \coth(\gamma a))}_{<0, \, \searrow} ,
\end{equation}
which is a negative and strictly decreasing function of $\gamma$ (in fact it is less than $-a$). For the other side of \eqref{eq:logA_2_b}, we compute
\begin{equation}
	\frac{h'(\gamma)}{h(\gamma)} = n a \coth\left(\tfrac{\gamma n a}{2}\right) \sech(\gamma n a).
\end{equation}

This is the product of two positive, strictly decreasing functions and hence positive and strictly decreasing. It follows that the right-hand side of \eqref{eq:logA_2_b} is negative and strictly increasing and therefore \eqref{eq:logA_2_b} has at most one solution. By considering the (directional) limits
\begin{equation}
	\lim_{x \to 0^+} \coth(x/2) \sech(x) = \infty
	\quad \text{ and } \quad
	\lim_{x \to \infty} \coth(x/2) \sech(x) = 0,
\end{equation}
we see that \eqref{eq:logA_2_b} must have exactly one solution and so the lower bound has a unique maximum (see Figure \ref{fig:bounds_2_b}). The optimal parameter $\gamma$, which is the unique solution of \eqref{eq:logA_2_b}, depends on the density $n$ of the lattice and the lattice geometry (the parameter $a$ fixes $b$). We note that, from a computational point of view, the decay parameter $\gamma$ and the lattice parameter $a$ are exchangeable. Hence, we also find that there is a unique lattice maximizing the lower frame bound. The lattice geometry, i.e., the parameter $a$, depends on the density and the decay parameter $\gamma$ and is (for any fixed $n$ and $\gamma$), the unique solution of $f'(a)/f(a) = - h'(a)/h(a)$.
\begin{figure}[ht]
	\subfigure[Lower frame bounds.]{
		\includegraphics[width=.45\textwidth]{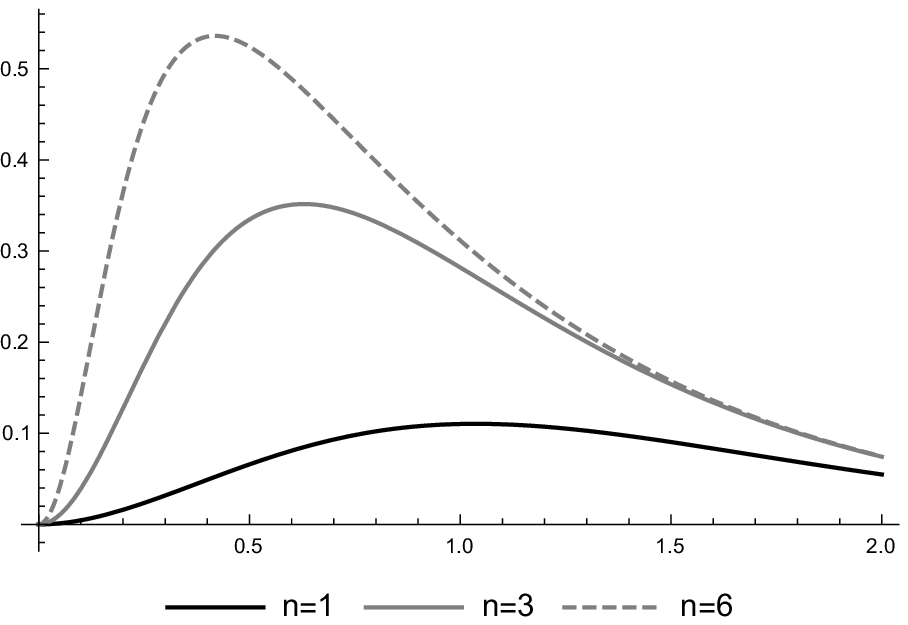}
	}
	\hfill
	\subfigure[Upper frame bounds.]{
		\includegraphics[width=.45\textwidth]{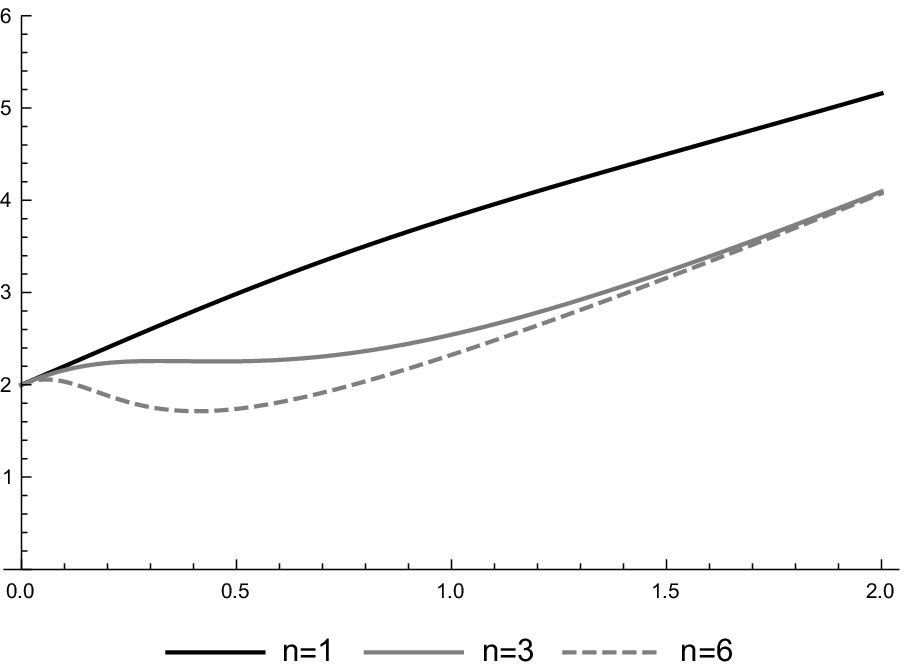}
	}
	\caption{Lower frame bounds and upper frame bounds in dependence of the lattice parameter $a$ for $\gamma=1$ and $n \in \{1,3,6\}$. We see that $A$ has a unique maximum, whereas $B$ may not have a minimizing lattice (the lattice is degenerated for $a=0$).}\label{fig:bounds_2_b}
\end{figure}

\subsubsection{The upper bound}
We will repeat the same kind of calculations for the upper frame bound. We re-write it as
\begin{equation}
	B = \underbrace{\gamma a (\coth(\gamma a) + 1)}_{\widetilde{f}(\gamma)} \, \underbrace{(1+\sech(\gamma n a))}_{\widetilde{h}(\gamma)}.
\end{equation}
Any critical points must satisfy the equation
\begin{equation}\label{eq:logB_2_b}
	\frac{\widetilde{f}'(\gamma)}{\widetilde{f}(\gamma)} = - \frac{\widetilde{h}'(\gamma)}{\widetilde{h}(\gamma)}.
\end{equation}
The left-hand side of \eqref{eq:logB_2_b} is independent of $n$ and is given by
\begin{equation}
	\frac{\widetilde{f}'(\gamma)}{\widetilde{f}(\gamma)} = a + \frac{1}{\gamma} - a \coth(\gamma a)
	= a + a \, \underbrace{\left( \frac{1}{\gamma a} - \coth(\gamma a) \right)}_{0>, \ \searrow \, , \ >-1},
\end{equation}
which is positive and decreasing (in fact, it is less than $a$). The right-hand side is given by
\begin{equation}
	 - \frac{\widetilde{h}'(\gamma)}{\widetilde{h}(\gamma)} = n a \sech(\gamma n a) \tanh \left( \tfrac{\gamma n a}{2} \right).
\end{equation}
This function is positive and we may use the logarithmic derivative to check the sign of the derivative. We obtain, after some simplifications by using identities for hyperbolic functions,
\begin{equation}
	\tfrac{d}{d \gamma} \left[\log\left(n a \sech(\gamma n a) \tanh \left( \tfrac{\gamma n a}{2} \right)\right)\right] = n a (\csch(\gamma n a) - \tanh(\gamma n a)),
\end{equation}
which has a unique 0. Since $-\widetilde{h}'(\gamma)/\widetilde{h}(\gamma)$ has only 1 critical point, we conclude that this is the unique global maximum by the asymptotic behavior. As the right-hand side of \eqref{eq:logB_2_b} decays by an exponential factor faster than the left-hand side, we conclude that, depending on $n$, \eqref{eq:logB_2_b} has either 0, 1 or 2 solutions (see Figure \ref{fig:supp_2_b_auxB}). We show that for $n \in \{1,2\}$, there is no solution and for $n \geq 3$ we find 2 solutions: dividing \eqref{eq:logB_2_b} by $a > 0$ leaves us with the equation
\begin{equation}\label{eq:supp_2_b_auxB}
	1 + \tfrac{1}{\gamma a} - \coth(\gamma a) = n \, \sech(\gamma n a) \tanh(\tfrac{\gamma n a}{2}),
\end{equation}
where the left-hand side is strictly decreasing and at most 1 and where the right-hand side has a unique global maximum. The point of the maximum is attained where the logarithmic derivative vanishes. This brings us to solving the equation
\begin{equation}
	\csch(\gamma n a) = \tanh(\gamma n a)
	\quad \Longleftrightarrow \quad
	\cosh(\gamma n a)^2 - \cosh(\gamma n a) - 1 = 0.
\end{equation}
The quadratic polynomial in $\cosh$ has exactly one positive solution, namely:
\begin{equation}
	\gamma n a = \arccosh\left(\tfrac{1 + \sqrt{5}}{2}\right) \approx 1.06128 \ldots
\end{equation}
\begin{figure}[ht]
	\includegraphics[width=.65\textwidth]{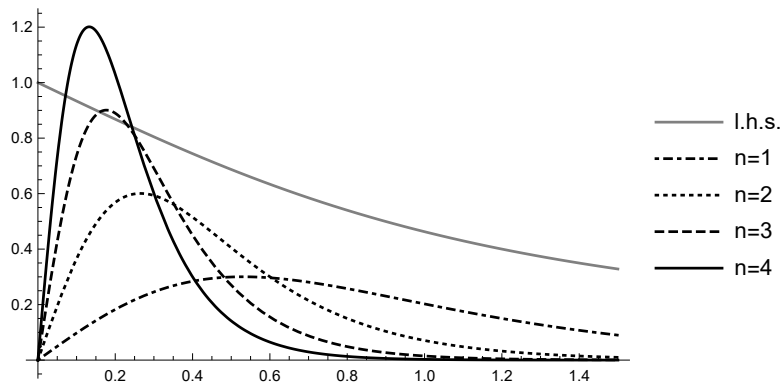}
	\caption{Illustration of the two sides of equation \eqref{eq:supp_2_b_auxB} for fixed $a=2$ and $n \in \{1,2,3,4\}$. The illustration suggests that for any $n \geq 3$ we have two solutions in $\gamma$ for equation \eqref{eq:supp_2_b_auxB}.}\label{fig:supp_2_b_auxB}
\end{figure}

We set $y = \gamma n a$ and to imply the correctness of \eqref{eq:supp_2_b_auxB} we want to determine those $n$ where
\begin{equation}
	n \, \sech(y) \tanh(\tfrac{y}{2}) > 1.
\end{equation}
Note that the solution depends on the product $\gamma n  a = y$ and so the optimal parameter $\gamma$ depends on the lattice geometry (the parameter $a$) and the lattice density (the parameter $n$). We evaluate $\sech(y) \tanh(y)$ close to its maximum, i.e., close to $y = \arccosh\left(\frac{1 + \sqrt{5}}{2}\right)$. To simplify computations, we set $y=1$, which is sufficiently close to the point of interest. After an elementary manipulation of the last inequality, this leaves us with solving
\begin{equation}
	n \tanh(\tfrac{1}{2}) > \cosh(1),
\end{equation}
which could be done numerically, but also with comparable effort analytically. We want to show that, for $n$ sufficiently large,
\begin{equation}
	n \tanh(\tfrac{1}{2}) = n - \frac{2n}{e+1} > \frac{e+e^{-1}}{2} = \cosh(1).
\end{equation}
It is not difficult to establish $e+1 > 3.7$ and $3.5 > e + e^{-1}$. So, it suffices to find $n$ such that
\begin{equation}
	2n - \tfrac{4n}{3.7} > 3.5,
\end{equation}
which holds for all $n \geq 4$. Hence, for $n \geq 4$ the upper bound $B$ (as a function of $\gamma$) has two critical points. By the behavior of the bound, the smaller value yields a local maximum and the larger value a minimum, which may be global, depending on $n$. For $n \in \{1,2,3\}$ we evaluate both sides of \eqref{eq:supp_2_b_auxB} with Mathematica at $\arccosh(\frac{1+\sqrt{5}}{2})$ and see that \eqref{eq:logB_2_b} does not have a solution for $n \in \{1,2\}$ and two solutions for $n=3$. Again, $\gamma$ may be exchanged for $a$. So, for $n \in \{1,2\}$ we find that there is no optimal lattice minimizing $B$ (the minimizing lattice degenerates) and for $n \geq 3$ there exist two critical lattices, one which is a local maximizer and the other a (local) minimizer, which is the only candidate for the global minimizer. As $B(0) = 2$, the minimizer is global if its value is less than 2 (compare Figure \ref{fig:bounds_2_b}).

\subsubsection{The condition number}
What is left to study is the condition number, which is simply given by the expression
\begin{equation}\label{eq:cond_2_b}
	\frac{B(\gamma)}{A(\gamma)} = e^{2 \gamma a} \coth\left(\frac{\gamma n a}{2}\right)^2.
\end{equation}
Taking the derivative with respect to $\gamma$ and looking for critical points leads to the equation
\begin{equation}
	\frac{2 a \, e^{2 \gamma a} \left(e^{2 n \gamma a} - n \, e^{n \gamma a} - 1\right)}{\left(e^{a \gamma  n}-1\right)^3} = 0
	\quad \Longleftrightarrow \quad
	e^{2 n \gamma a} - n \, e^{n \gamma a} - 1 = 0.
\end{equation}

\begin{wrapfigure}[14]{r}{.5\textwidth}
	\vspace*{-0.6cm}
		\includegraphics[width=.475\textwidth]{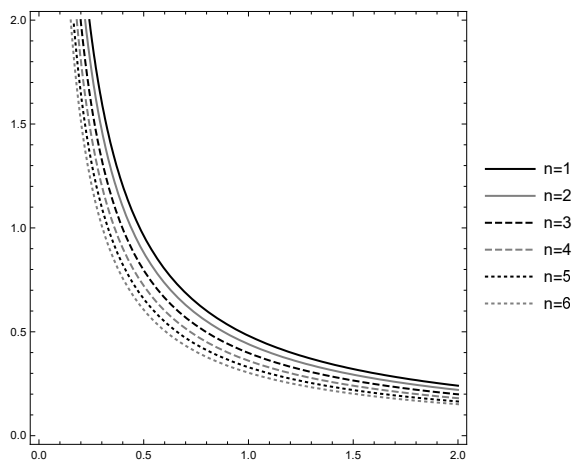}
	\vspace*{-0.3cm}
	\caption{Curves of optimal parameters $(a,\gamma)$ for different lattice densities $n \in \{1,2,3,4,5,6\}$.}\label{fig:cutoff_2b_opt_lattice}
\end{wrapfigure}
Setting $x = e^{n \gamma a}$ leaves us with solving a simple quadratic equation: $x^2-nx-1 = 0$. There is only one positive solution to this equation, which is $x = (n + \sqrt{n^2+4})/2$. Therefore, the critical points of the condition number are attained for
\begin{equation}\label{eq:curves}
	\hspace*{-0.5\textwidth}
	\gamma = \frac{\log\left(n+\sqrt{n^2+4}\right)-\log(2)}{a \, n}.
\end{equation}
Note that the dependence on the lattice parameter $a$ is the same as for the decay parameter $\gamma$, which can easily be seen from \eqref{eq:cond_2_b}. Moreover, we see that \eqref{eq:curves} describes a curve in the space of parameters $(a,\gamma)$ yielding the optimal pair of decay and lattice parameters in dependence of the oversampling rate $n$ (see Figure \ref{fig:cutoff_2b_opt_lattice}). Lastly, we deduce from \eqref{eq:cond_2_b} that $B/A$ is decreasing as a function of $n \in \N$ and
\begin{equation}
	\frac{B}{A} \to e^{2 \gamma a}
	\quad \text{ for } \quad
	n \to \infty.
\end{equation}
So, for general $\gamma > 0$ and non-degenerate lattice ($a > 0$), the frame operator does not converge to a (multiple of) the identity operator, which however holds for any window $g \in S_0(\R)$ \cite{FeiKoz98}.

This fact seems to be accompanied by the results in \cite{FeiJan00} (compare also \cite[p.\ 132]{Gro01}). We remark that exchanging the parameters $a$ and $\gamma$ comes from the algebraic expressions in the calculations and is not a dilation result by applying a metaplectic operator.

\section{The one-sided exponential}\label{sec:1sided}
In this section we study frame bounds of one-sided exponentials. The lattice parameters satisfy $(ab)^{-1} = n\in \N$. In particular, we obtain a frame at the critical density as shown by Janssen \cite[Sec.~4]{Jan96}. The function of interest is 
\begin{equation}
g_\gamma(t) = \sqrt{2\gamma}\, e^{-\gamma t}\indicator_{(0,\infty)}(t),\quad t\in\R,
\end{equation}
where $\gamma>0$ is the dilation parameter. By Proposition \ref{prop:symplectic_invariance}, it suffices to consider the case $\gamma=1$. Considering the dependencies as presented in \cite{Jan96}, we will express the lattice pair again as $(\eta/n, 1/\eta)$, $\eta>0$. Then the frame bounds are illustrated in Figure \ref{fig:1sided_A_B} and given by
\begin{equation} 
	A(\tfrac{\eta}{n}, \tfrac{1}{\eta}) = 2\eta \, \tanh(\tfrac{\eta}{2})\, \csch(\tfrac{\eta}{n}) \, e^{- \tfrac{\eta}{n}} \qquad \text{and} \qquad 
	B(\tfrac{\eta}{n}, \tfrac{1}{\eta}) = 2\eta \, \coth(\tfrac{\eta}{2})\, \csch(\tfrac{\eta}{n}) \, e^{ \tfrac{\eta}{n}}. 
\end{equation}
\begin{figure}[h!t]
	\subfigure[Lower frame bounds.]{
		\includegraphics[width=.45\textwidth]{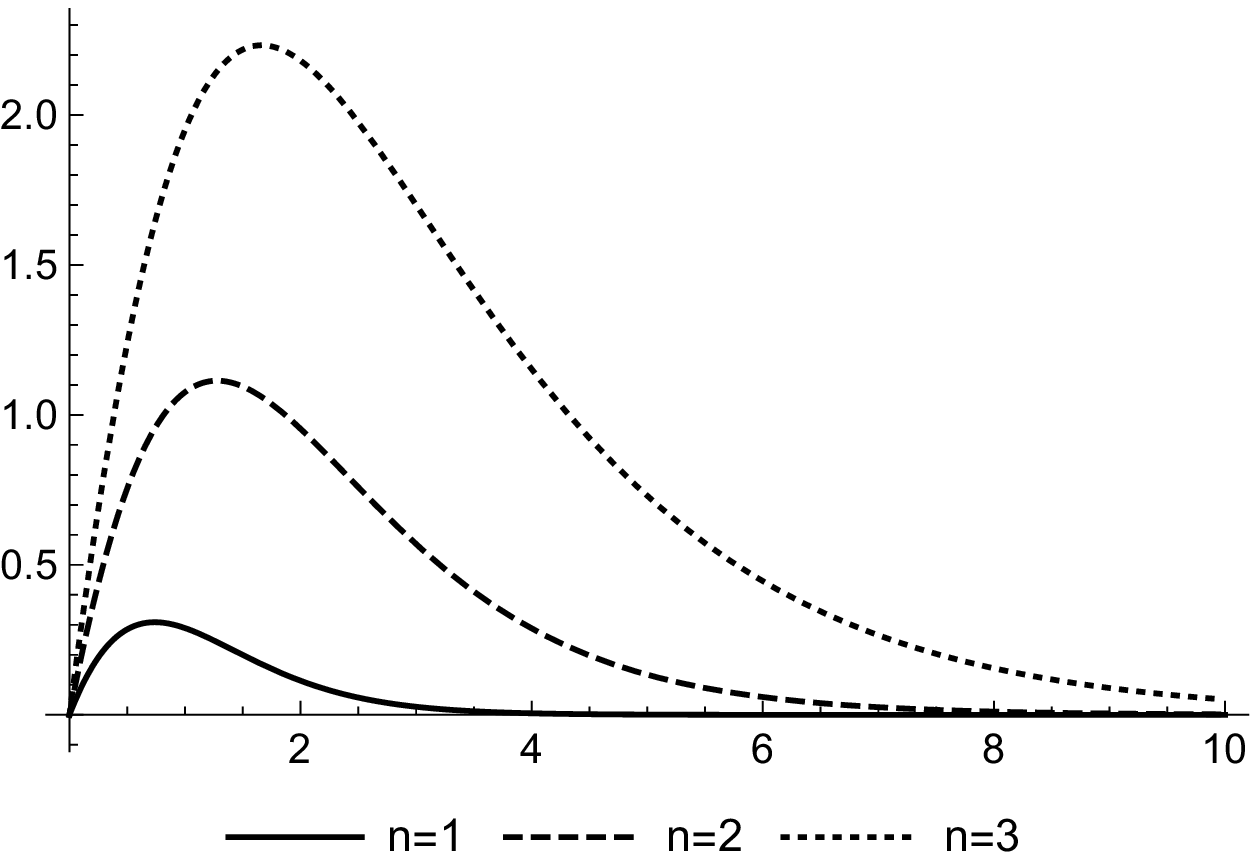}
	}
	\hfill
	\subfigure[Upper frame bounds.]{
		\includegraphics[width=.45\textwidth]{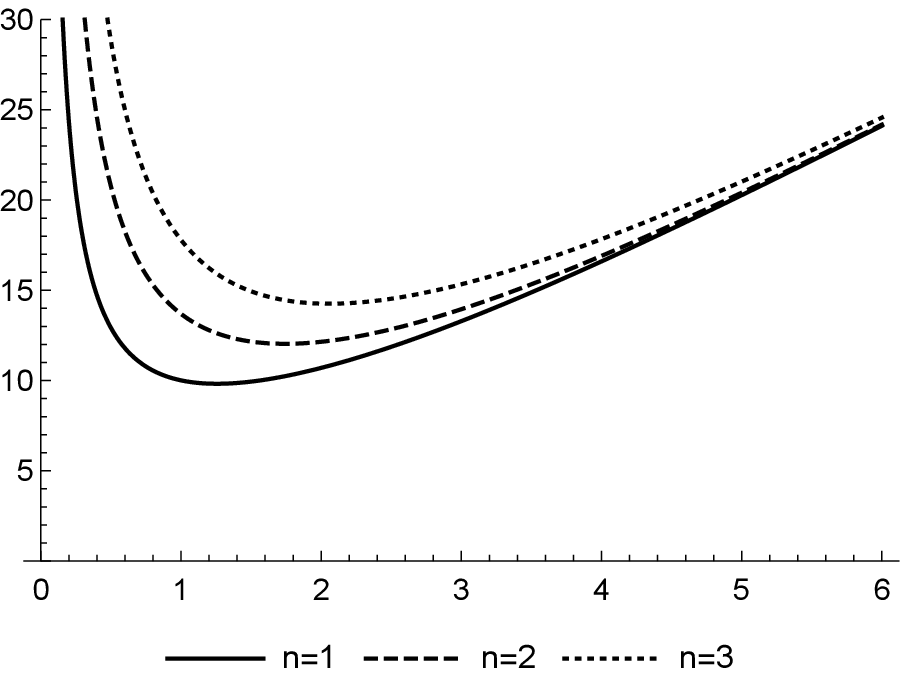}
	}
	\hfill
	\caption{Lower and upper bounds for the parameter $\gamma=1$ and density $n\in\{1, 2, 3\}$.}\label{fig:1sided_A_B}
\end{figure}

\subsection{The condition number}\label{sec:1sided_cond}
This time, we start with the condition number which has an obvious minimizing lattice. We easily compute
\begin{equation}
	\kappa(\tfrac{\eta}{n}, \tfrac{1}{\eta}) = \big(\coth(\tfrac{\eta}{2})\, e^{\frac{\eta}{n}}\big)^2,
\end{equation}
which is minimal if and only if $\sqrt{\kappa}$ is minimal. We compute the derivative of $\sqrt{\kappa}$:
\begin{equation}
	\tfrac{d}{d \eta}\big(\coth(\tfrac{\eta}{2})\, e^{\frac{\eta}{n}}\big)  = \tfrac{1}{n} e^{\frac{\eta}{n }} \csch^2(\tfrac{\eta}{2})\ (\sinh(\eta)-n).
\end{equation}
The sign of the derivative is determined by $\sinh$, which is a strictly monotonically increasing, unbounded, and attains $0$ at $0$. Therefore, the unique minimum of the condition number $\kappa$ is attained at $\eta = \arcsinh(n)$. Next, we will show that this point lies between the maximizer of the lower and the minimizer of the upper bound.

\subsection{The lower bound}
For $\gamma = 1$, we start with computing the first derivative of the lower bound $A=A_1$ and show that it has a unique zero.
\begin{align}
	\tfrac{d}{d \eta}\ A(\tfrac{\eta}{n}, \tfrac{1}{\eta}) 
		&=e^{-\frac{\eta}{n}}\csch(\tfrac{\eta}{n}) \tanh(\tfrac{\eta}{2})\Big( 1+\eta\,\csch(\eta)-\tfrac{\eta}{n}-\tfrac{\eta}{n}\coth(\tfrac{\eta}{n}) \Big),
\end{align}
where the last equality is due to the addition theorems\footnote{\textsuperscript{, \scriptsize{9}, \scriptsize{10}} \suppfoot}. Furthermore, the function 
\begin{equation}
	\eta \mapsto 1+\eta\,\csch(\eta)-\tfrac{\eta}{n}-\tfrac{\eta}{n}\coth(\tfrac{\eta}{n}), \qquad n \in \N,
\end{equation}
is a sum of strictly monotonically decreasing functions ($f$ and $-x-x\coth(x)$). Taking into account the monotonicity in $n$, 
it suffices to verify there is a $\eta\in(0,\infty)$ with $A(\eta/n,\eta^{-1})>0$. One valid choice is $\eta=0.5$. 
On the other hand, the function is unbounded from below:
\begin{align}
	\lim\limits_{\eta\rightarrow\infty}\eta \csch(\eta)-\tfrac{\eta}{n}-\tfrac{\eta}{n}\coth(\tfrac{\eta}{n}) = -\infty.
\end{align}
Therefore, there exists a unique $\eta_{A,n} > 0$ such that the lower bound $A(\tfrac{\eta}{n}, \tfrac{1}{\eta})$ has a unique maximum in $\eta_{A,n}$. We check whether it is $\eta=\arcsinh(n)$, which we know optimizes the condition number:
\begin{equation}
	1+\eta \csch(\eta)-\tfrac{\eta}{n}-\tfrac{\eta}{n}\coth(\tfrac{\eta}{n}) = 1+\tfrac{\eta}{n}-\tfrac{\eta}{n}-\tfrac{\eta}{n}\coth(\tfrac{\eta}{n}) = 1-\tfrac{\eta}{n}\coth(\tfrac{\eta}{n})<0.
\end{equation}
It follows that $\eta_{A,n}<\arcsinh(n)$. It is also simple to see that $(\eta_{A,n})_{n\in\N}$ is strictly monotonically increasing: $\eta_{A,n}$ is the unique solution to 
\begin{equation}
	1+\eta \csch(\eta)= \tfrac{\eta}{n}+\tfrac{\eta}{n}\coth(\tfrac{\eta}{n}). 
\end{equation}
The left-hand side is independent of $n$ whereas the right-hand side is strictly decreasing in $n$.

\subsection{The upper bound}
We deal with the upper bound in the same manner and set $B = B_1$.
\begin{align}
	\tfrac{d}{d \eta}\ B(\tfrac{\eta}{n}, \tfrac{1}{\eta}) 
	&=e^{\frac{\eta}{n}}\csch(\tfrac{\eta}{n}) \coth(\tfrac{\eta}{2})\Big( 1+\tfrac{\eta}{n}-\tfrac{\eta}{n}\coth(\tfrac{\eta}{n}) -\eta\,\csch(\eta)\Big)
\end{align}
due to the addition theorems\footnotemark. We again analyze the expression in the brackets.
To begin with, $- x\csch(x)$ is a strictly increasing function.
A curve discussion\footnotemark 
\ and employing the strict monotonicity of $x\coth(x)$, and  $x\coth(x)\to 1$ for $x\to 0$,
shows that $1+x-x\coth(x)$ is a strictly concave, strictly monotonically increasing function. We now ensure the existence of a unique critical point of $B$:
\begin{equation}
\lim\limits_{\eta\rightarrow 0^+}1+\tfrac{\eta}{n}-\tfrac{\eta}{n}\coth(\tfrac{\eta}{n})-\eta\,\csch(\eta) = 1+0-1-1 = -1 <0,
\end{equation}
If we evaluate the same expression at $2n$, we get
\begin{equation}
1-2n\,\csch(2n)+\tfrac{2n}{n}-\tfrac{2n}{n}\coth(\tfrac{2n}{n}) \geq 3-2\coth(2)-2\csch(2) >0.
\end{equation}
We could again check whether $\eta = \arcsinh(n)$ is a critical point. From the general theory, though, we already know that it has to be between $\eta_{A,n}$ and $\eta_{B,n}$ and we have already established that it is greater than $\eta_{A,n}$.

\section{The two-sided exponential}\label{sec:2sided}
The last section is dedicated to the two-sided exponential. At first  sight, the bounds look quite similar to those already analyzed. However, this one has proven to be the most challenging to inspect. The current window function is
\begin{equation}
g_\gamma(t) = \sqrt{\gamma}\, e^{-\gamma |t|}, \quad t\in\R,
\end{equation}
where $\gamma>0$ is the dilation parameter. By Proposition \ref{prop:symplectic_invariance}, it suffices to consider the case $\gamma=1$. Unlike the one-sided exponential, the two-sided one is in Feichtinger's algebra, so we do not have a frame at the critical density. It has been shown in \cite{Jan96} that we obtain a frame for densities larger than 1. With the same parametrization $(\eta/n, 1/\eta)$ as in the previous sections, the frame bounds are given by
\begin{align}\label{eq:2sided_AB}
	\begin{split}
		A(\tfrac{\eta}{n},\tfrac{1}{\eta}) & = \tanh(\tfrac{\eta}{2})\Big(\tfrac{\eta}{n}\,\csch(\tfrac{\eta}{n})- \eta \, \csch(\eta)\Big) , \\
		B(\tfrac{\eta}{n},\tfrac{1}{\eta}) & = \coth(\tfrac{\eta}{2})\Big( \tfrac{\eta}{n}\,\coth(\tfrac{\eta}{n})+\eta\,\csch(\eta)\Big).
	\end{split}
\end{align}
\begin{figure}[h!tp]
	\subfigure[Lower frame bounds.]{
		\includegraphics[width=.45\textwidth]{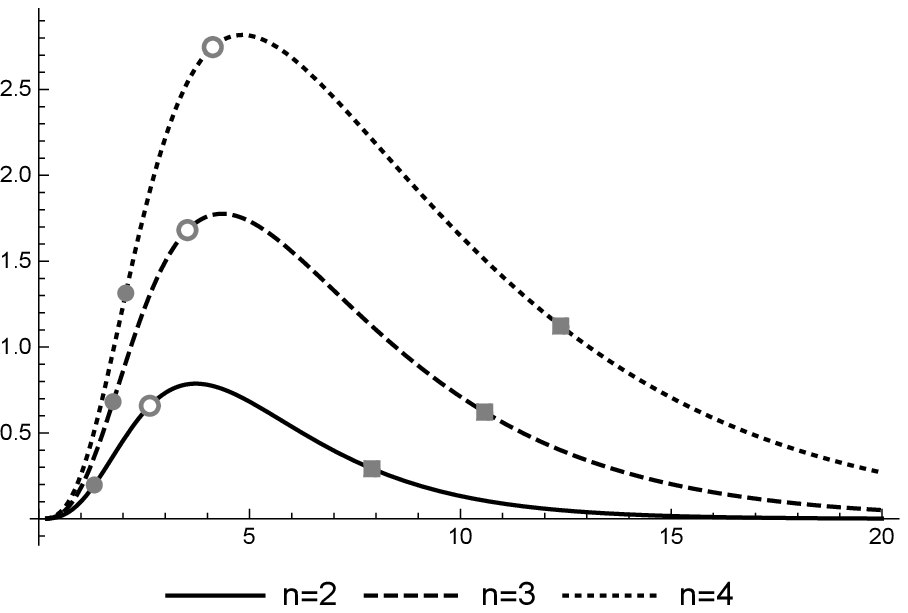}
	}
	\hfill
	\subfigure[Upper frame bounds.]{
		\includegraphics[width=.45\textwidth]{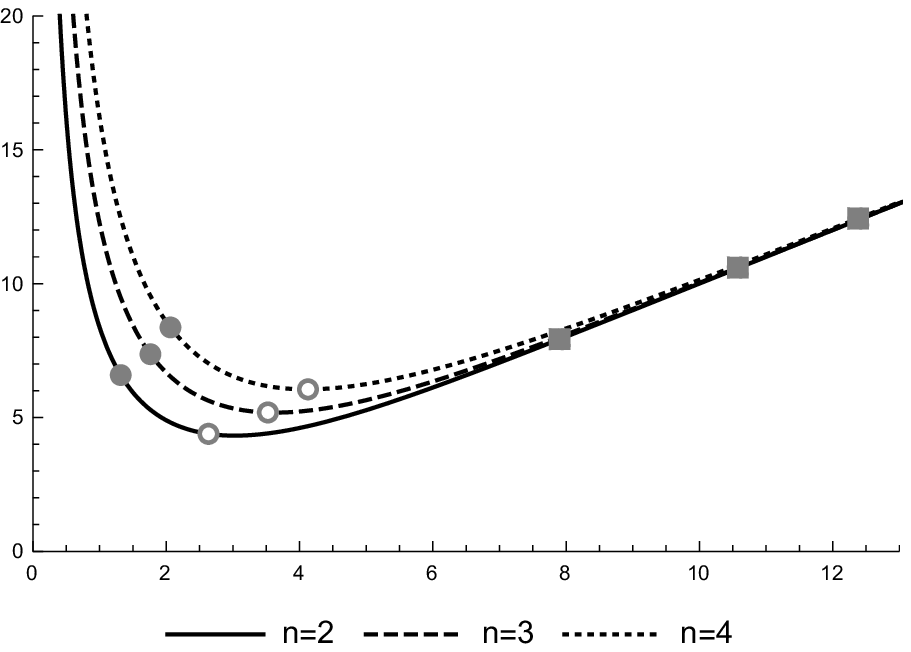}
	}
	\caption{The idea of the cut-offs for the lower and upper bound and parameter $\gamma=1$ and densities $n\in\{ 2, 3,4\}$. The $\textcolor{gray}{\bullet}$ and $\textcolor{gray}{\blacksquare}$ mark left and right cut-offs, respectively and $\textcolor{gray}{\circ}$ marks~$\eta_n$.}\label{fig:2sided_A_B}
\end{figure}

When one looks at the graphs of the lower or the upper bound for a fixed density in Figure \ref{fig:2sided_A_B}, it seems obvious that there is a unique extremum. However, the first and second derivative of the concerned bound are somewhat difficult to analyze over $(0,\infty)$, as they are combinations of products and sums which exhibit different behaviors at different intervals. The visually ``clear" situation is due to the interaction altogether, so dissecting each part separately leads to unsatisfactory estimates. Therefore, to determine the global extrema and to show their uniqueness, we will first exclude most of the positive half-axis through interval estimates (see Section \ref{sec:intervals} and Figure \ref{fig:2sided_A_B}) as yielding potential global extrema. Then for the lower bound, upper bound and condition number we will establish certain concavity, convexity and $\log$-convexity results, respectively, on the small remaining interval. These will allow us to conclude that the extremum is unique in each case. For the lower bound, the cases for small $n$, which is $3\leq n\leq 10$, and in particular $n=3$, will need more attention at some points compared to large $n$. The case $n=2$ needs to be handled completely separately, since it does not agree with the asymptotic positions of the extremal points. In the sequel, we will set
\begin{equation}\label{eq:2sided_f}
	f(t)=t \csch(t), \quad \forall t > 0,
\end{equation}
so as to have more compact expressions and we will repeatedly use the following:
$f$ is strictly monotonically decreasing and asymptotically satisfies\footnote{\textsuperscript{, \scriptsize{12}} \suppfoot}
\begin{equation}
 \lim_{x \to 0} x \csch(x) =1 \quad\text{and}\quad
 \lim_{x \to \infty} x \csch(x) = 0.
\end{equation}

\subsection{The lower bound}
The initial reference point will be $\eta_n=2\arccosh(n)$. We will first compare the value attained here with the function's values on $(0,\eta_n/2]$ and $[2\eta_n,\infty)$ in order to exclude global maxima in these intervals. Then we will show that there is a unique maximum in $(\eta_n/2, 2 \eta_n)$. By using \eqref{eq:2sided_f} and addition formulas, we write the lower frame bound evaluated at $\eta_n = 2 \arccosh(n)$ as 
\begin{equation}\label{eq:2sided_A_eta_n} 
	A(\tfrac{\eta_n}{n},\tfrac{1}{\eta_n}) = \frac{(n^2-1)^{1/2}}{n} \Big(f(\tfrac{\eta_n}{n}) - \frac{\eta_n}{2n (n^2-1)^{1/2}} \Big) = \frac{(n^2-1)^{1/2}}{n} f(\tfrac{\eta_n}{n}) - \frac{\arccosh(n)}{n^2}. 
\end{equation} 

\subsubsection{The left cut-off}
We will show that for any $3 \leq n \in \N$ the value of the lower bound achieved at $\eta_n = 2 \arccosh(n)$ is larger than any value attained in $(0, \eta_n/2]$. For small $n$ we will later only be able to show a concavity result for slightly smaller intervals (depending on $n$). The needed intervals are given below and we use interval estimates, as presented in Section \ref{sec:intervals}, to show that the maximum is still not attained on the mildly larger interval.

Before we begin with the comparison, we introduce a second function that will accompany us till the end:
\begin{equation}\label{eq:rho_n}
\rho(x)\coloneqq \frac{2\arccosh(x)}{x},\qquad x>0.
\end{equation}
This is a strictly monotonically decreasing function on $[2,\infty)$, which can easily be seen when computing $\rho'(x)$\footnotemark. Using the expression from \eqref{eq:2sided_AB} for $A$ evaluated at $\eta_n$ in combination with the notation \eqref{eq:2sided_f}, we have for all $\eta\leq \eta_n/2$ that
\begin{align}
	A(\tfrac{\eta}{n}, \tfrac{1}{\eta}) &< \tanh({\eta_n}/{4})\,\big(1-f({\eta_n}/{2})\big)
	=\Big(1 - \tfrac{2}{n+1}  \Big)^{1/2}  \,\Big(1 - \tfrac{\eta_n}{2(n^2-1)^{1/2}}\Big).
\end{align}
Using \eqref{eq:2sided_A_eta_n}, 
the desired inequality is
\begin{align}
	&\ \Big(1 - \frac{2}{n+1}  \Big)^{1/2}  \,\Big(1 - \frac{\eta_n}{2(n^2-1)^{1/2}}\Big)
	< \frac{(n^2-1)^{1/2}}{n}f(\tfrac{\eta_n}{n}) - \frac{\arccosh(n)}{n^2} \\
	\Longleftrightarrow &\ \Big(1-\frac{2}{n+1} \Big)^{1/2} - \frac{(n^2-1)^{1/2}}{n}f(\tfrac{\eta_n}{n})
	< \arccosh(n)\,\Big(\,\Big(1-\frac{2}{n+1}  \Big)^{1/2}\tfrac{1}{(n^2-1)^{1/2}}- \tfrac{1}{n^2}\Big).
\end{align}
We claim that the left-hand side is negative and the right-hand side is positive. We start with the negativity of the left-hand side:
\begin{equation}
	\Big(1 - \frac{2}{n+1}  \Big)^{1/2} -  \frac{(n^2-1)^{1/2}}{n}f(\tfrac{\eta_n}{n})\ <\,0 \qquad 
	\Longleftrightarrow \qquad\ \frac{n}{n+1}- f(\tfrac{\eta_n}{n})<0.
\end{equation}
For all $n\geq 3$, it will suffice to show
\begin{equation}
	1-\frac{3}{4 n}< f(\tfrac{\eta_n}{n}) \qquad
	\Longleftrightarrow  \qquad\Big( 1-\frac{3}{4 n}\Big) \frac{n}{\eta_n} - \csch(\tfrac{\eta_n}{n}) <0.
\end{equation}
We claim that the left side of the last inequality is strictly monotonically increasing in $n$. Recall that $\rho(n) = \eta_n/n$ is strictly monotonically decreasing on $[2,\infty)$. We differentiate the left-hand side of the last inequality with respect to $n$:
\begin{align}
	\frac{d}{d n} \left[\Big( 1-\frac{3}{4 n}\Big) \frac{n}{\eta_n} - \csch(\tfrac{\eta_n}{n}) \right]
	= \frac{3}{4 n\,\eta_n}+\frac{\eta_n - 2n \, (n^2-1)^{-1/2}}{\eta_n^2}\ \Big( \frac{\eta_n^2}{n^2} \csch'(\tfrac{\eta_n}{n})+1 - \frac{3}{4n}\Big).
\end{align}
The only part whose positivity has not been proven stands in the brackets. Keeping in mind that $\eta_n/n = \rho(n)\leq \rho(3) = \eta_3/3\leq 1.2$, we obtain
\begin{equation}
	\frac{\eta_n^2}{n^2} \csch'(\tfrac{\eta_n}{n})+1 - \frac{3}{4n} = \underbrace{1- \frac{3}{4n}}_{\nearrow} - \underbrace{\rho(n)^2 \csch(\rho(n))}_{\searrow}\ \underbrace{\left(\rho(n) \coth(\rho(n))-1 \right)}_{\searrow}.
\end{equation}
The atomic parts of the expression are all positive, so in total, the expression is strictly monotonically increasing. Evaluating at $n=3$ shows that the expression is always positive, therefore $\Big( 1-\tfrac{3}{4 n}\Big) \tfrac{n}{\eta_n} - \csch(\tfrac{\eta_n}{n}) $ is strictly monotonically increasing in $n$. Due to the fact that $1/t - \csch(t) \to 0$, for $t \to 0$, 
we have
\begin{equation}
	\lim\limits_{n\rightarrow \infty}\left[\Big( 1-\frac{3}{4 n}\Big) \frac{n}{\eta_n} - \csch(\tfrac{\eta_n}{n})\right] = 
	\lim\limits_{n\rightarrow \infty} \left[\frac{n}{\eta_n} - \csch(\tfrac{\eta_n}{n})  - \frac{3}{4 \eta_n} \right] = 0,
\end{equation}
implying that the desired expression is negative. We now examine the right-hand side of the initial inequality. As $\arccosh(x) > 0$ for all $x > 1$, we only need to establish
\begin{equation}
	\Big(1 - \frac{2}{n+1}  \Big)^{1/2}\frac{1}{(n^2-1)^{1/2}} \ > \ \frac{1}{n^2}
	\qquad \Longleftrightarrow \qquad
	\frac{1}{n+1}-\frac{1}{n^2}>0, \quad n\geq 3.
\end{equation}
All in all, we have shown that the maximum of $A$ cannot be attained in the interval $(0,\eta_n/2]$.

Now, as mentioned above, we will need to improve our estimates for $n=3, \ldots ,10$ in order to show the necessary concavity/convexity later on.
\begin{itemize}
	\item ($4\leq n\leq 10$): we compare the value of $\tanh(\eta/2)(1-f(\eta))$ (which majorizes to the lower bound) at $\eta = 3.08$ with that of $A(\eta_n/n,1/\eta_n)/n$. We see that there is no maximum in $(0,3.08] \supset (0, \eta_n/2]$.
	\item ($n=3$): this case is a bit more subtle. Here, we perform an interval estimate on $(1.7, 3.3)\supset(\eta_3/2,3.3)$ with sub-intervals of length $0.3$ to see that the global maximum does not lie in the interval $(0, 3.3]$.
\end{itemize}

\subsubsection{The right cut-off}
Similarly to the left cut-off, we will show that for $3 \leq n \in \N$ the lower bound $A$ does not assume its global maximum on $[2\eta_n, \infty)$. This is again achieved by comparing with the value at $\eta_n$ and with extra estimates for small $n$. By using the properties of $f(t)$, defined by Equation \eqref{eq:2sided_f}, and the fact that $\tanh(\eta) < 1$, we have
\begin{equation}
	A(\tfrac{\eta}{n}, \tfrac{1}{\eta}) < f(\tfrac{2\eta_n}{n}) = \tfrac{2\eta_n}{n} \csch(\tfrac{2\eta_n}{n}) = \tfrac{\eta_n}{n} \csch(\tfrac{\eta_n}{n})\sech(\tfrac{\eta_n}{n}) =f(\tfrac{\eta_n}{n})\sech(\tfrac{\eta_n}{n}) .
\end{equation}
As already introduced Equation \eqref{eq:rho_n}, we use $\rho(n)=\frac{2\arccosh(n)}{n}$. This leaves us with showing the following inequality\footnote{\ \suppfoot}:
\begin{align}
	& \frac{(n^2-1)^{1/2}}{n}f(\rho(n)) - \frac{\rho(n)}{2n} > f(\rho(n))\sech(\rho(n))  \\
	\Longleftrightarrow \qquad &
	(n^2-1)^{1/2}\cosh(\rho(n))-n > \frac{\sinh(2\rho(n))}{4}.
\end{align}
We want to show that
\begin{equation}
	(n^2-1)^{1/2}\cosh(\rho(n))-n - \tfrac{\sinh(2\rho(n))}{4} >0.
\end{equation}
We first assume that $n\geq 10$ and bound $\cosh$ from below and $\sinh$ from above. For $\sinh$, we note that $2\rho(n)\leq 2\rho(10) <1.2$ (as $n \geq 10$), and use $\sinh(x)<x+x^3$ for $x\in[0,1.2]$:
\begin{align}
	& \, (n^2-1)^{1/2}\cosh(\rho(n))-n - \tfrac{\sinh(2\rho(n))}{4} >
	(n^2-1)^{1/2}\Big( 1+ \tfrac{\rho(n))^2}{2}\Big)-n - \tfrac{2\rho(n)+(2\rho(n))^3}{4} \\
	= & \, (n^2-1)^{1/2} - n +(n^2-1)^{1/2}\tfrac{\rho(n))^2}{2} - \rho(n)-4\rho(n)^3. 
\end{align}
We claim that $(n^2-1)^{1/2}-n > -\tfrac{1}{n}$ and that the remaining part dominates $\tfrac{1}{n}$. The former can be verified with curve discussion\footnote{\ \suppfoot}.
We are left to prove
\begin{align}
	\tfrac{1}{n} & < \rho(n) \Big( (n^2-1)^{1/2}\tfrac{\rho(n))}{2} - 1-4\rho(n)^2\Big)\\
	& = \tfrac{2\arccosh(n)}{n} \Big( \tfrac{(n^2-1)^{1/2}}{n}\arccosh(n)) - 1-4\rho(n)^2\Big) \\
	\Longleftrightarrow
	1 & < 2 \underbrace{\arccosh(n)}_{\nearrow, \, >0}\Big( \underbrace{\Big(1-\tfrac{1}{n^2}\Big)^{1/2}}_{\nearrow, \, >0} \underbrace{\arccosh(n)}_{\nearrow, \, >0} - 1- \underbrace{4\rho(n)^2}_{\searrow, \, >0}\Big).
\end{align}
In summary, for $n \geq 10$, the expression in the brackets is strictly monotonically increasing. Evaluating at $10$ shows that it is also positive for all $n\geq 10$. Therefore, the right-hand side of the above inequality is strictly monotonically increasing. Finally, the evaluation at $n =10$ also shows the claimed inequality. For the initial $3\leq n\leq 9$ we evaluated the initial inequality to verify it.

Again, the intervals where we look for the maximum need to be updated for $n\leq 10$.
\begin{itemize}
	\item ($4 \leq n \leq 10$): we ensure that there is no global maximum in $[4n/3, 2 \eta_n)$. One can check this with a direct interval estimate for $5\leq n\leq 10$. In the case $n=4$ we split it in the intervals $[16/3, 5)$ and $[5, 2 \eta_4)$.
	\item ($n=3$): we divide $[4.8, 7.1) \supseteq[4.8, 2 \eta_3)$ in sub-intervals of length $0.1$ to show that there is no global maximum in this interval.
\end{itemize}	
Summing up all the cases from the left cut-off and right cut-off, we are left to show that there is a unique (global) maximum in the interval (see also Figure \ref{fig:I_n})
\begin{equation}\label{eq:intervals}
	I_n = \left(\max\{3.08, \tfrac{\eta_n}{2}\}, \min\{\tfrac{4n}{3},2\eta_n\}\right), \quad 4 \leq n \in \N
	\quad \text{ or } \quad
	I_n = (3.3, 4.8), \quad n = 3.
\end{equation}
\begin{figure}[h!tp]
	\includegraphics[width=.65\textwidth]{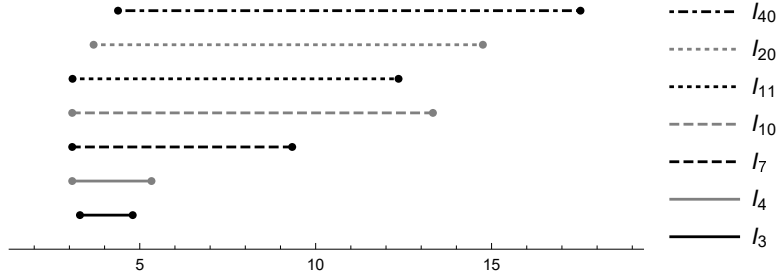}
	\caption{Illustration of the intervals $I_n$ for $n \in \{3,4,7,10,11,20,40\}$. We have the intervals $I_3=(3.3,4.8)$, $I_n=(3.08, 4n/3)$ for $4 \leq n \leq 10$ and $I_n = (\eta_n/2, 2 \eta_n)$ for $n \geq 11$.}\label{fig:I_n}
\end{figure}

\subsubsection{The concavity result}
We fix $n\geq 3$ and set $\psi_n(\eta) = \psi(\eta) = f(\eta/n)-f(\eta)$, where $f(t) = t \csch(t)$ is still defined by \eqref{eq:2sided_f}. The lower frame bound can then be expressed as
\begin{equation}
	 A(\tfrac{\eta}{n}, \tfrac{1}{\eta}) = \tanh(\tfrac{\eta}{2}) \, \psi(\eta).
\end{equation}
We want to show that $A$ has a unique maximum in $(\eta_n/2, 2 \eta_n)$. As $A$ is strictly positive, we may perform the analysis for $\log(A)$, which has the same critical points. We compute
\begin{equation}
	\tfrac{d^2}{d \eta^2} \left[ \log\left(
	A(\tfrac{\eta}{n},\tfrac{1}{\eta})\right) \right] = \frac{\tanh''(\tfrac{\eta}{2}) \tanh(\tfrac{\eta}{2}) - \tanh'(\tfrac{\eta}{2})^2}{4 \tanh(\tfrac{\eta}{2})^2} + \frac{\psi''(\eta) \psi(\eta) - \psi'(\eta)^2}{\psi(\eta)^2}.
\end{equation}
Since, $\tanh$ is positive, strictly increasing and strictly concave and $\psi$ is positive ($f$ is strictly monotonically decreasing), it suffices to show that $\psi$ is concave in order to have the whole expression negative. We compute its second derivative, which we may compactly write as
\begin{align}
	\quad\psi''(\eta) = \tfrac{1}{n^2}f''(\tfrac{\eta}{n}) -  f''(\eta), \quad \eta \in (\eta_n/2, 2 \eta_n).
\end{align}
Thus, we turn back to the properties of $f$, more specifically, 
$f''(t)$. It can be shown\footnote{\ \suppfoot} that $f''$ is strictly negative on $(0, 1.6]$ and strictly positive on $[1.61, \infty)$.  Now, recall that
\begin{equation}
	\eta \in (\eta_n/2, 2 \eta_n)
	\quad \text{ with } \quad
	\eta_n = 2 \arccosh(n).
\end{equation}
We observe that ${\eta_n}/{2} \geq {\eta_3}/{2} > 1.76$. As a consequence, we have that $f''(\eta) > 0$ for all $\eta \geq \eta_n/2 > 1.61$. Next, recall that $\rho(x) = 2 \arccosh(x)/x$ is strictly decreasing for $x > 2$. Also, $\eta/n < 2\rho(n)$ and $2\rho(7) < 1.51$. In particular, $f''(\eta/n) < 0$ for all $\eta \leq 2 \eta_n$ and $n \geq 7$. Therefore, $\psi''(\eta) <0 $ for all $n \geq 7$. We have shown, for $n\geq 7$, that:
\begin{itemize}
\item $A(\eta_n/n, 1/\eta_n) > A(\eta/n,1/\eta)$, for all $\eta \in (0,\eta_n/2] \cup [2\eta_n, \infty)$.
	\item $\log(A(\eta/n,1/\eta))$ is strictly concave in $\eta$ on $(\eta_n/2, 2 \eta_n)$ and therefore has a unique maximum on this interval. As $\log(A(\eta/n,1/\eta))$ and $A(\eta/n,1/\eta)$ have maxima at the same positions, it follows that $A(\eta/n,1/\eta)$ has a unique maximum in $(\eta_n/2, 2 \eta_n)$.
\end{itemize}

The fact that $\eta/n$ needs to be less than 1.6 makes it necessary to shorten the intervals appropriately for $3 \leq n \leq 6$. This is achieved by using the intervals defined in \eqref{eq:intervals}. To that, due to the position of the maximum, this does not work for $n=2$.

All in all, for $n \geq 3$ the lower frame bound has a unique global maximum in $\eta_{{}_{A,n}} \in (\eta_n/2, 2\eta_n)$.

\subsection{The upper bound}\label{sec:2sided_B}
We carry on with the same tactic, i.e., we evaluate $B$ given by \eqref{eq:2sided_AB} at $\eta_n/2$, $\eta_n$ and $2\eta_n$ and obtain a convexity result to see that the minimum lies in $(\eta_n/2, 2 \eta_n)$.
\subsubsection{The convexity}\label{sec:2sided_B_convex}
This time we start with the convexity of the bound, which actually holds on $\R_+$. It then follows from the asymptotic behavior that we have a unique minimum.
\begin{align} 
	B(\tfrac{\eta}{n},\tfrac{1}{\eta}) &>  \, (\tfrac{\eta}{n} \coth (\tfrac{\eta}{n}) +0) \to \infty, \quad \eta \to \infty,\\
	B(\tfrac{\eta}{n},\tfrac{1}{\eta}) &> \coth(\tfrac{\eta}{2}) \, (1 +0) \to \infty, \quad \eta \to 0^+.
\end{align}
We may already include the case $n=2$ this time. For a fixed $n \geq 2$, we set 
\begin{equation}\label{eq:2sided_B_aux_h}
	h(\eta) = \tfrac{\eta}{n} \coth(\tfrac{\eta}{n}) + \eta \csch(\eta).
\end{equation}
The first two derivatives of $h$ are given by\footnote{\textsuperscript{, \scriptsize{17}, \scriptsize{18}} \suppfoot}
\vspace*{-0.6cm}
\begin{align}
	h'(\eta) & = \tfrac{h(\eta)}{\eta} - \eta\,\Big( \tfrac{1}{n^2}\csch(\tfrac{\eta}{n})^2+\csch(\eta)\coth(\eta)\Big),\\[-10pt]
	h''(\eta) & = - 2\big(\tfrac{1}{n^2}\csch(\tfrac{\eta}{n})^2+\csch(\eta)\coth(\eta) \big) + \eta\big(\tfrac{2}{n^3} \csch(\tfrac{\eta}{n})^2\coth(\tfrac{\eta}{n}) +\csch(\eta)+2\csch(\eta)^3\big).
\end{align}
We apply the Leibniz rule and obtain\footnotemark
\begin{align}
	\tfrac{d^2}{d \eta^2} \, \left[B(\tfrac{\eta}{n},\tfrac{1}{\eta}) \right]
	& =\tfrac{1}{2}\coth(\tfrac{\eta}{2})\csch(\tfrac{\eta}{2})^2\Big(\tfrac{\eta}{n} \coth(\tfrac{\eta}{n}) + \eta \csch(\eta)\Big) \\
	& \quad -\csch(\tfrac{\eta}{2})^2\Big(\tfrac{1}{n} \coth(\tfrac{\eta}{n}) + \csch(\eta) -\eta\big(\tfrac{1}{n^2}\csch(\tfrac{\eta}{n})^2+\csch(\eta)\coth(\eta) \big)\Big) \\
	& \quad +\coth(\tfrac{\eta}{2}) \Big(- 2\big(\tfrac{1}{n^2}\csch(\tfrac{\eta}{n})^2+\csch(\eta)\coth(\eta) \big) + \\
	& \quad \phantom{+\coth(\tfrac{\eta}{2})} \quad \quad \eta\big(\tfrac{2}{n^3} \csch(\tfrac{\eta}{n})^2\coth(\tfrac{\eta}{n}) +\csch(\eta)+2\csch(\eta)^3\big)\Big).
\end{align}
First, we deal with the parts independent of $n$. For uniformity, all hyperbolic functions will be transformed to an argument $\eta/2$ by using addition formulas. 
Set
\begin{align}
	T(\eta) & = \csch(\tfrac{\eta}{2})^2 \cdot \Big(\tfrac{1}{2}\coth(\tfrac{\eta}{2})\eta \csch(\eta)-\csch(\eta) +\eta\csch(\eta)\coth(\eta) \Big)\\
	& \quad + \coth(\tfrac{\eta}{2}) \Big(- 2\csch(\eta)\coth(\eta) +\eta\csch(\eta)+2\eta\csch(\eta)^3\Big).
\end{align}
With addition formulas for hyperbolic functions, this can be written as\footnotemark
\begin{align}
	T(\eta)	&  = \tfrac{1}{4} \csch( \tfrac{\eta}{2})^2 \, \cdot \eta \csch(\tfrac{\eta}{2})^2\cdot \sum\limits_{k=1}^\infty \eta^{2k} \tfrac{2k-1}{(2k+1)!} >0.
\end{align}
So, the part independent of $n$ is positive. If we can show the same for the remaining terms of the second derivative, then we have proven the strict convexity of the bound on its entire domain. We collect the terms dependent on $n$ and set
\begin{align}
	T_n(\eta) & = \tfrac{1}{2}\cdot\tfrac{\eta}{n}\coth(\tfrac{\eta}{2})\csch(\tfrac{\eta}{2})^2 \coth(\tfrac{\eta}{n})-\tfrac{1}{n} \csch(\tfrac{\eta}{2})^2 \coth(\tfrac{\eta}{n}) 
	+\tfrac{\eta}{n^2}\csch(\tfrac{\eta}{2})^2\csch(\tfrac{\eta}{n})^2 \\
	& \quad - \tfrac{2}{n^2}\coth(\tfrac{\eta}{2})\csch(\tfrac{\eta}{n})^2 +\tfrac{2\eta}{n^3}\coth(\tfrac{\eta}{2}) \csch(\tfrac{\eta}{n})^2\coth(\tfrac{\eta}{n}).
\end{align}
This can be written as
\begin{align}
	T_n(\eta) & =\underbrace{\tfrac{1}{n}\csch(\tfrac{\eta}{2})^2 \coth(\tfrac{\eta}{n})}_{>0} \underbrace{\Big(\tfrac{\eta}{2}\coth(\tfrac{\eta}{2})-1\Big)}_{>0}+\underbrace{\tfrac{\eta}{n^2}\csch(\tfrac{\eta}{2})^2\csch(\tfrac{\eta}{n})^2}_{>0}\\
	& \quad + \underbrace{\tfrac{2}{n^2}\coth(\tfrac{\eta}{2})\csch(\tfrac{\eta}{n})^2}_{>0}\underbrace{\Big(\tfrac{\eta}{n}   \coth(\tfrac{\eta}{n}) -1\Big)}_{>0}>0.
\end{align}
All in all, the upper frame bound is strictly convex and tends to infinity when approaching the boundary of its domain, which is $\R_+$. Therefore, it has a unique minimizer $\eta_{{}_{B,n}}$, as claimed. By comparing the values of $B$ for $\eta \in \{\eta_n/2, \eta_n, 2\eta_n\}$, one sees that $\eta_{B,n}$ has to lie in $I_n$. Since the points are chosen to allow for the addition formulas, the desired inequalities
\begin{equation}
	B(\tfrac{\eta_n}{n},\tfrac{1}{\eta_n})< B(\tfrac{\eta_n}{2n},\tfrac{2}{\eta_n}) \qquad \text{and}\qquad B(\tfrac{\eta_n}{n},\tfrac{1}{\eta_n})< 	B(\tfrac{2\eta_n}{n},\tfrac{1}{2\eta_n}) 
\end{equation} 
are, after a few equivalence transformations, reduced to easily verifiable inequalities and we omitted them from the manuscript\footnote{\textsuperscript{, \scriptsize{20}} \suppfoot}. 
We remark already that for the first $3\leq n\leq 10$ we will use the intervals $I_n$ defined in \eqref{eq:intervals} in order to obtain the needed concavity results for the condition number. It will suffice to evaluate the first derivative on the boundaries of $I_n$.

\subsection{The condition number}
Unlike for all other windows, things seem to get only more complicated with the condition number. It is given by
\begin{equation}
	\kappa(\tfrac{\eta}{n}, \tfrac{1}{\eta}) = \frac{B(\frac{\eta}{n})}{A(\frac{\eta}{n})} = \coth(\tfrac{\eta}{2})^2 \, \frac{\tfrac{\eta}{n}\coth(\tfrac{\eta}{n}) +\eta\csch(\eta)}{\tfrac{\eta}{n} \csch(\tfrac{\eta}{n})-\eta \csch(\eta)}.
\end{equation}
We try to use maximally what we already know and to simplify the problem of determining an optimal lattice as much as possible. By the asymptotic behaviour of $A$ and $B$, we know that $\kappa$ is unbounded towards 0 and $\infty$. Showing strict convexity would therefore ensure that we have a unique minimum. This turns out to be a rather difficult problem due to the algebraic nature of the condition number. We simplify the problem by showing the strict log-convexity on $I_n$, defined by \eqref{eq:intervals}, and that there is no minimum outside. We compute
\begin{equation}
	\log(\kappa(\tfrac{\eta}{n}, \tfrac{1}{\eta})) = -\log(A(\tfrac{\eta}{n}, \tfrac{1}{\eta})) + \log(B(\tfrac{\eta}{n}, \tfrac{1}{\eta})) = - \log( A(\tfrac{\eta}{n}, \tfrac{1}{\eta})) - 2\log(\tanh(\tfrac{\eta}{2})) + \log(h(\eta)),
\end{equation}
where, for $3 \leq n \in \N$ fixed,
\begin{equation}
	h(\eta) = \tfrac{\eta}{n} \coth(\tfrac{\eta}{n}) + \eta \csch(\eta),
\end{equation}
is still defined by Equation \eqref{eq:2sided_B_aux_h}, just as in the proof of the convexity of the upper bound. We know that $\tanh(x)$ is strictly log-concave for $x>0$ and that the lower bound is strictly log-concave on $I_n$. If we can show that $h$ is log-convex on $I_n$, then we have proven that $\kappa$ is strictly log-convex, hence strictly convex, on $I_n$. As $A$ does not have a maximum outside $I_n$ and $B$ does not possess a minimum there, the minimum in $I_n$ has to be global and, by convexity, unique. So, we want to show 
\begin{equation}
	\tfrac{d^2}{d \eta^2} \log(h(\eta)) > 0, \quad \eta \in I_n
	\quad \Longleftrightarrow \quad
	h(\eta) h''(\eta) - h'(\eta)^2 >0, \quad \eta \in I_n.
\end{equation}

We continue by using the explicit formulas for $h$, $h'$ and $h''$ and, then further simplifying the expression by grouping terms and using the addition theorems for hyperbolic functions. 
The computations need very good intuition of how to group the terms efficiently, making 
this part rather laborious and quite lengthy. However, as each step only needs elementary manipulations of hyperbolic functions, we omit the details at this point\footnotemark.
The result to further analyze is
\vspace*{-0.6cm}
\begin{align}
	h(\eta) h''(\eta) - h'(\eta)^2
	& = \frac{\csch(\tfrac{\eta}{n})^2}{n^2} \underbrace{ \left( (\tfrac{\eta}{n})^2\coth(\tfrac{\eta}{n})^2 - \cosh(\tfrac{\eta}{n})^2 +(\tfrac{\eta}{n})^2 \right)}_{(1)}\\  
	& \quad + \tfrac{1}{n} \coth(\tfrac{\eta}{n}) \csch(\eta) \Big(\eta \underbrace{\big(\eta-\tfrac{2\eta}{n} \coth(\tfrac{\eta}{n})\coth(\eta)\big)}_{(2)} +2\underbrace{\big((\tfrac{\eta}{n})^2\coth(\tfrac{\eta}{n})^2-1\big)}_{>0} \Big) \\ 
	& \quad + \tfrac{2\eta}{n}\csch(\eta)\Big(
	\tfrac{\eta\csch(\eta)}{2}\underbrace{\big( \cosh(\eta) - \tfrac{n}{\eta^2} \big)}_{(3)} +\underbrace{\tfrac{\eta}{2}\coth(\eta) - \tfrac{\eta}{n}\coth(\tfrac{\eta}{n})}_{(4)}\Big)\\
	& \quad + \underbrace{\csch(\eta)^2 \big(\tfrac{2\eta^2}{n}\coth(\tfrac{\eta}{n})\csch(\eta)+\eta^2\csch(\eta)^2\big) }_{>0}.
\end{align}
Assuming $n\geq 4$, one can check in a straightforward manner\footnote{\textsuperscript{, \scriptsize{22}, \scriptsize{23}} \suppfoot} that all of the expressions (1) -- (4) in the above calculation are positive on $I_n$.

\subsubsection{The case $n=3$}
 
For the interval $I_3 = (3.3, 4.8)$ we make the following three claims:
\begin{equation}
	h \geq 1.6
	\quad \text{ and } \quad
	|h'| \leq 0.22
	\quad \text{ and } \quad
	h'' > 0.
\end{equation}
They can easily be verified\footnotemark\ with interval estimates on subintervals of $I_3$ of length $0.1$.
With these properties of $h$,
we can estimate as follows:
\begin{align}
h(\eta)h''(\eta) - h'(\eta)^2 & > 1.6 h''(\eta) - 0.22^2 \\
	& = \tfrac{3.2}{9}\csch(\tfrac{\eta}{3})^2\big(\tfrac{\eta}{3}\coth(\tfrac{\eta}{3})-1\big)\\
	& \quad +3.2\csch(\eta)\big(\eta\csch(\eta)^2+\tfrac{\eta}{2}-\coth(\eta)\big) -0.0484\\
	& > \tfrac{3.2}{9} \underbrace{\tfrac{\eta}{3}\csch(\tfrac{\eta}{3})\cdot \csch(\tfrac{\eta}{3})\coth(\tfrac{\eta}{3})}_{\searrow } - \tfrac{3.2}{9}\cdot 0.64\, \underbrace{\csch(\eta)}_{\searrow}-0.0484>0.
\end{align}
The final estimate can be confirmed by an interval estimate on subintervals of length $0.1$. In conclusion, the condition number is strictly $\log$-convex, hence strictly convex, on $(3.3, 4.8)$, where the minimum lies. Therefore, the minimum is unique on in this interval and by what we know about the bounds $A$ and $B$ outside this interval, the minimum has to be global.

\subsection{The case \texorpdfstring{$n=2$}{n=2}}
This case needs to be treated separately as the necessary adjustments would go beyond simple interval shrinkage.
On the other hand, this case has an additional symmetry, as the arguments $\eta/2$ and $\eta/n$ now coincide. This allows us to manipulate the relevant functions more easily with the addition theorems for hyperbolic functions.
Since the expressions to analyze are simplifications of those already considered above, we omit the estimates\footnotemark\ and summarize the important observations. The lower bound attains its unique maximum $\eta_{A,2}$ in $(3.7, 3.74)$, the upper bound attains its minimum $\eta_{B,2}$ in $(3, 3.1)$ and the condition number is uniquely minimizes between these extremal points: $\eta_{B,2}<\eta_{2}<\eta_{A,2}$.

\subsection{Three optimal points}
Lastly, we want to show that the locations of the optimum for the lower bound, upper bound and condition number differ from each other. As we already know that the optimum is unique in each case, it suffices to show that $A$ and $B$ do not assume their optima at the same point. By recalling \eqref{eq:A'=B'}, it is immediate that also $\kappa$ cannot have its optimum at one of these positions.

\subsubsection{The lower bound} 
We are now in a position to determine the location of the maximum a bit better. We will show that $\eta_{{}_{A,n}} >\eta_n$. It suffices to show that the first derivative of $A$ is positive at $\eta_{n}$. Since $f$ is strictly monotonically decreasing, we have
\begin{align}
	 \tfrac{d}{d \eta}A(\tfrac{\eta}{n}, \tfrac{1}{\eta})\Bigr|_{\eta=\eta_n}& =  \tfrac{1}{2} \sech(\tfrac{\eta_n}{2})^2\underbrace{\big( f(\tfrac{\eta_n}{n}) - f(\eta)\big)}_{>0} + \tanh(\tfrac{\eta_n}{2})\big(\tfrac{1}{n}f'(\tfrac{\eta_n}{n}) - f'(\eta_n)\big) \\
	& > \tanh(\tfrac{\eta_n}{2})\big(\tfrac{1}{n}f'(\tfrac{\eta_n}{n}) - f'(\eta_n)\big).
\end{align}
By the definition of $f$ (Equation \eqref{eq:2sided_f}), this is strictly positive if and only if
\begin{align}
	0 & < n\csch(\tfrac{\eta_n}{n})\big(1-\tfrac{\eta_n}{n}\coth(\tfrac{\eta_n}{n})\big)- n^2\csch(\eta_n)(1-\eta_n\coth(\eta_n)).
\end{align}
We make several observations. First of all, the addition formulas can be applied wherever the term $\eta_n = 2 \arccosh(n)$ appears. Secondly, the parts involving $\eta_n/n$ are not as straightforward to analyze, so we transform the parameter to $t = \rho(n) =\eta_n/n$ and use that it is strictly monotonically decreasing in $n\geq 2$. 
We further manipulate the right-hand side of the above inequality with the mentioned addition formulas and substitutions to an equivalent inequality;
\begin{align}
	0&< n t\big(\tfrac{\csch(t)}{t}-\csch(t)\coth(t)+\tfrac{1}{2}\big) -\tfrac{1}{2}\big(1+ \tfrac{1}{ n^2-1}\big)^{1/2}+ \tfrac{n\, t}{4(n^2-1)}
= (\otimes).
\end{align}
As usual, we decompose the expression $(\otimes)$ into parts;
\begin{align}
	\tfrac{\csch(t)}{t}-\csch(t)\coth(t)+\tfrac{1}{2}
	& >\csch(t)^2 \big(\tfrac{\sinh(t)}{t} -1\big)
	= \tfrac{\csch(t)}{t} - \csch(t)^2.\\
\end{align}
We will now show that the last expression is strictly monotonically decreasing for all $t >0$;
\begin{align}
	\tfrac{d}{d t } \left[\tfrac{\csch(t)}{t} - \csch(t)^2 \right]	& = -\tfrac{\csch(t)^4}{t^2}\,
	\sum\limits_{k=2}^\infty\tfrac{4\,t^{2k+2}}{(2k+2)!}
	\Big((k+2)(2^{2k-1} - 2k-1) +(2k+1) \Big).
\end{align}
It is not hard to convince oneself of the positivity of
\begin{equation}
	(k+2)(2^{2k-1} - 2k-1) +(2k+1), \quad k \geq 2,
\end{equation} 
through the behavior of $2^{x-1}-x$ for $x \geq 4$. Therefore, 
\begin{align}
	\tfrac{\csch(\rho(n))}{\rho(n)} - \csch(\rho(n))^2 \geq \tfrac{\csch(\rho(2))}{\rho(2)} - \csch(\rho(2))^2 > 0, \quad \forall n \geq 2.
\end{align}
So, $\csch(t)/t-\csch(t)^2$ is decreasing for $t > 0$, as claimed. We go back to proving positivity of the expression $(\otimes)$.
\begin{align}
	(\otimes)
	& > n t \csch(t) \big(\tfrac{1}{t} -\csch(t)\big) -\tfrac{1}{2}\big(1+\tfrac{1}{ n^2-1}\big)^{1/2} \\
	&= \underbrace{2\arccosh(n)}_{\nearrow}\underbrace{\big(\tfrac{ \csch(\rho(n)) }{\rho(n)} - \csch(\rho(n)) ^2\big)}_{>0,\,\nearrow} \underbrace{-\tfrac{1}{2}\big(1+ \tfrac{1}{ n^2-1}\big)^{1/2}}_{ \nearrow}.
\end{align}
We evaluate at $n =5$ to verify the positivity of $(\otimes)$ for all $n \geq 5$ by the monotonicity of all terms as indicated above. For the initial $n=2,3,4$, we actually evaluate the derivative of $A$ at $\eta_n$ and conclude that 
\begin{equation}
	\tfrac{d}{d \eta}A(\tfrac{\eta}{n}, \tfrac{1}{\eta})\Bigr|_{\eta=\eta_n} >0
\end{equation}
for all $n\geq 2$. This allows us to conclude
$\eta_{{}_n}<\eta_{{}_{A,n}}$ for all $n\geq 2$.

\subsubsection{The upper bound}
We finalize by showing that $\eta_{{}_{B,n}}\neq \eta_{{}_{A,n}}$. In, fact we show a stronger statement: $\eta_{{}_{B,n}}<\eta_{{}_n}$ for $n\geq 4$. 
Using the calculations for $h$ as defined in \eqref{eq:2sided_B_aux_h}, and $\eta_n = n\rho (n) = 2\arccosh(n)$, we compute\footnote{\ \suppfoot}
\begin{align}\label{eq:3points_aux_B}
	\tfrac{d}{d \eta} 
	B(\tfrac{\eta}{n}, \tfrac{1}{n}) \Bigr|_{\eta=\eta_n}
	& = \tfrac{1}{2\,(n^2-1)} \big(1 -\tfrac{\eta_n}{n} \tfrac{1}{2(n^2-1)^{1/2}} -\tfrac{\eta_n}{n} \coth(\tfrac{\eta_n}{n}) \big)\\
	& \quad + \tfrac{1}{(n^2-1)^{1/2}}
	\bigg(
		\underbrace{\coth(\tfrac{\eta_n}{n})	- \tfrac{\eta_n}{n}\csch(\tfrac{\eta_n}{n})^2 -\tfrac{\eta_n}{2n}}_{=\text{(I)}} + \underbrace{\tfrac{\eta_n}{n}\tfrac{1}{4\,(n^2-1)}}_{\geq 0}
	\bigg).
\end{align}
If the last expression, and hence the derivative of the upper frame bound at $\eta_n$, is positive, then  we can  conclude that $\eta_{{}_{B,n}} < \eta_n$. We will first prove this for $n \geq 4$.

We start with two substitutions in (I), namely $\csch^2 = \coth^2-1$ and $t = \eta_n/n$;
\begin{align}
	\coth(t) - t \csch(t )^2 -\tfrac{t }{2} +\tfrac{t }{4\,(n^2-1)} > \coth(t )- t \coth(t )^2 +\tfrac{t }{2} -\tfrac{t }{14}, \quad t > 0, \, n \geq 2.
\end{align}
We will now study the properties of this new lower bound. We compute\footnote{\textsuperscript{, \scriptsize{26}, \scriptsize{27}} \suppfoot}
\begin{align}
	\tfrac{d^2}{d t^2} \left[ \coth(t)- t \coth(t)^2 +\tfrac{t}{2} - \tfrac{t}{14}\right]
	&=-\csch(t)^4\,\sum\limits_{k=1}^\infty2k \, \tfrac{(2t)^{2k+1}}{(2k+1)!} < 0, \quad \forall t > 0.
\end{align}
This shows that the expression $\coth(t )- t\coth(t )^2 +\tfrac{t}{2} - \tfrac{t}{14}$ is strictly concave on $\R_+$. We observe that, $\coth(1.05)- 1.05\coth(1.05)^2 +\tfrac{3\cdot 1.05}{7} >0$. Additionally\footnotemark,
\begin{align}
	\lim\limits_{t\rightarrow 0^+}\coth(t)- t\coth(t )^2 +\frac{t}{2} - \frac{t}{14} 
	& = \lim\limits_{t\rightarrow 0^+}\ \sum\limits_{k=1}^{\infty}\tfrac{2^{2k+1}\,(1-k)}{(2k+1)!}t^{2k-1} = 0,
\end{align}
due to the convergence of the series. So, on $(0, 1.05)\supseteq (0, \rho(n))$ and for $n\geq 4$, this implies 
\begin{equation}
	\coth(t )- t \coth(t )^2 +\tfrac{t }{2} > \tfrac{t}{14}.
\end{equation}
We 
establish\footnotemark\ a lower bound for \eqref{eq:3points_aux_B} and $n \geq 4$.
\begin{align}
	\tfrac{d}{d \eta} 
	B(\tfrac{\eta}{n}, \tfrac{1}{n}) \Bigr|_{\eta=\eta_n}
	&> \tfrac{1}{2\,(n^2-1)} \big(1 -\underbrace{\tfrac{\rho(n)}{2(n^2-1)^{1/2}}}_{\searrow} -\underbrace{\rho(n)\coth(\rho(n))}_{\searrow} +\underbrace{\big(1-\tfrac{1}{n^2}\big)^{1/2}}_{\nearrow}\underbrace{\tfrac{2\arccosh(n)}{7}}_{\nearrow}\big).
\end{align}

By evaluating at $n=4$, we obtain
\begin{align}
	\tfrac{d}{d \eta} 
	B(\tfrac{\eta}{n}, \tfrac{1}{n}) \Bigr|_{\eta=\eta_n}
	&>\tfrac{1}{2\,(n^2-1)} \big(1 -\tfrac{\rho(4)}{2(4^2-1)^{1/2}} -\rho(4)\coth(\rho(4)) +\big(1-\tfrac{1}{4^2}\big)^{1/2} \, \tfrac{2\arccosh(4)}{7}\big) >0.
\end{align}

Hence, $\eta_{{}_{B,n}}< \eta_n <\eta_{{}_{A,n}}$ for all $n \geq 4$.
For $n=2,3$ we have $\eta_{{}_{B,n}}>\eta_n$. For $n=2$, we have already established that $\eta_{{}_{B,2}} <3.2 < \eta_{{}_{A,2}}$. As far as $n=3$ is concerned, 
\begin{equation}
	\tfrac{d}{d \eta}A(\tfrac{\eta}{3}, \tfrac{1}{\eta})\Bigr|_{\substack{\eta=3.9}} >0 
	\quad \text{ and } \quad
	\tfrac{d}{d \eta}B(\tfrac{\eta}{3}, \tfrac{1}{\eta})\Bigr|_{\substack{\eta=3.9}} >0,
\end{equation}
finally confirming that
\begin{equation}
	\eta_{{}_{B,n}}<\eta_{{}_{A,n}}, \quad \forall n \geq 2, \, n \in \N.
\end{equation}

\end{document}